\documentclass[12pt]{article}

\usepackage[utf8]{inputenc}
\usepackage[T1]{fontenc}
\usepackage{bbm}
\usepackage{stmaryrd}
\usepackage{latexsym}
\usepackage{amsfonts}
\usepackage{amsmath,amssymb,amscd,textcomp}
\usepackage{yfonts}
\usepackage{dsfont}
\usepackage{mathabx}
\usepackage[dvips]{graphicx}
\usepackage{epsfig}
\usepackage{psfrag}
\usepackage[font={small}]{caption}
\usepackage{color}
\usepackage{float}
\usepackage{upgreek}
\usepackage{tikz}
\usepackage{tikz-cd}
\usepackage{relsize}
\usepackage{enumerate} 
 \usepackage[all]{xy}
 \usepackage{hyperref}
\DeclareFontFamily{U}{txsyc}{}
\DeclareFontShape{U}{txsyc}{m}{n}{
   <-> txsyc%
}{}
\DeclareFontShape{U}{txsyc}{bx}{n}{
   <-> txbsyc%
}{}
\DeclareFontShape{U}{txsyc}{l}{n}{<->ssub * txsyc/m/n}{}
\DeclareFontShape{U}{txsyc}{b}{n}{<->ssub * txsyc/bx/n}{}
\DeclareSymbolFont{symbolsC}{U}{txsyc}{m}{n}
\SetSymbolFont{symbolsC}{bold}{U}{txsyc}{bx}{n}
\DeclareFontSubstitution{U}{txsyc}{m}{n}
\DeclareMathSymbol{\df}{\mathrel}{symbolsC}{"42}
\DeclareMathSymbol{\fd}{\mathrel}{symbolsC}{"43}
\DeclareMathSymbol{\lJoin}{\mathrel}{symbolsC}{"58}
\DeclareMathSymbol{\rJoin}{\mathrel}{symbolsC}{"59}

\newcommand{\f}[2]{\frac{#1}{#2}}

\newcommand{\cB}{\mathcal{B}}

\newcommand{\cL}{\mathcal{L}}

\newcommand{\cU}{\mathcal{U}}

\newcommand{\EE}{\mathbb{E}}

\newcommand{\NN}{\mathbb{N}}
\newcommand{\PP}{\mathbb{P}}

\newcommand{\usm}{\underline{\smash{\mu}}}

\newcommand{\ZZ}{\mathbb{Z}}

\newcommand{\fs}{\mathfrak{s}}

\newcommand{\iy}{\infty}

\newcommand{\lt}{\left}
\newcommand{\me}{\medskip}

\newcommand{\pa}{\partial}
\newcommand{\ri}{\rightarrow}
\newcommand{\rt}{\right}
\newcommand{\sm}{\smallskip}

\newcommand{\wi}{\widetilde}

\DeclareMathOperator*{\esssup}{ess\,sup}

\newcommand{\fo}{\forall\ }

\newcommand{\st}{\,:\,}

\newcommand{\un}{\mathds{1}}

\newcommand{\bq}{\begin{eqnarray*}}
\newcommand{\bqn}[1]{\begin{eqnarray}\label{#1}}
\newcommand{\eq}{\end{eqnarray*}}
\newcommand{\eqn}{\end{eqnarray}}
\newcommand{\wwtbp}{\par\hfill $\blacksquare$\par\me\noindent}

\newcommand{\thistitlepagestyle}{}

\newcommand{\ttsim}{\raise.17ex\hbox{$\scriptstyle\mathtt{\sim}$}}

\newcommand{\kh}{\kern .08em}
\newcommand{\bb}{\mathbb{B}}
\newcommand{\eql}{\stackrel{\mathcal{L}}{=}}

\newtheorem{pro}{Proposition} 
\newtheorem{cor}[pro]{Corollary}
\newtheorem{lem}[pro]{Lemma}
\newtheorem{theo}[pro]{Theorem}

\newtheorem{defi}[pro]{Definition}
\renewcommand{\thepro}{\arabic{pro}}

\newenvironment{deff}
{\par\me\refstepcounter{pro}\noindent{\bf Definition \thepro\ }}
{\par\hfill $\square$\par\sm\noindent}

\newenvironment{rem}
{\par\me\refstepcounter{pro}\noindent{\bf Remark \thepro\ }}
{\par\hfill $\square$\par\sm\noindent}

\newcommand{\proof}{\par\me\noindent\textbf{Proof}\par\sm\noindent}

\newcommand{\comment}[1]{}

\allowdisplaybreaks

\title{Cutoff in separation profile for the flat torus, sphere, and projective spaces.}
 \author{
 Kol\'eh\`e Coulibaly-Pasquier}

 \date{\vbox{\copy0
 \copy1
 \copy2
}
 }

\begin{document}


\setbox1=\vbox{
\large
\begin{center}
 Institut \'Elie Cartan de Lorraine, UMR 7502\\
Universit\'e de Lorraine and CNRS
\end{center}
} 

\setbox5=\vbox{
\hbox{kolehe.coulibaly@univ-lorraine.fr\\[1mm]}
\hbox{Institut \'Elie Cartan de Lorraine\\}
\hbox{Universit\'e de Lorraine}
}



\maketitle
\thistitlepagestyle
\abstract{ In this paper we show that the cutoff in separation profile for  Brownian motion on flat torus $ \mathbb{T}^n$; on spheres  $\mathbb{S}^n $; on  real, complex and quaternionic projective space  resp. $ \mathbb{P}^n(\mathbb{R})$, $ \mathbb{P}^n(\mathbb{C})$ and $ \mathbb{P}^n(\mathbb{H})$, is the tail distribution of  some explicit Gumbel distribution. The
proof is based on intertwining,  dual process together with a representation formula  of large moments of the covering time of the dual process.}

\vfill\null
{\small
\textbf{Keywords: }Brownian motions on Riemannian manifolds, intertwining relations, set-valued dual processes, couplings of primal and dual processes, stochastic mean curvature evolutions, strong stationary times, separation discrepancy, hitting times, Cutoff profile.
\par
\vskip.3cm
\textbf{MSC2010:} primary: 58J65, secondary:  37A25 58J35 60J60 35K08.
\par\vskip.3cm

 \textbf{Fundings:} the grants ANR-25-CE40-6875-01.
}\par


\section{Introduction}

 Consider a family of Brownian motion $X^n $ in a family of compact Riemannian  manifolds $M^n$ of real dimension $n$ starting at some point. For fixed $ n$, this diffusion process approaches the uniform distribution $\mathcal{U}_{M^n} $ overs $ M^n$ as $t$ goes to infinity. In this paper we are interested in  family of manifolds $M^n$ which are flat torus, spheres, real complex and quaternionic projective spaces. A natural question is to estimate the corresponding speeds of convergence,
or mixing times,  for large $n$, and it depends on how the distance between the
time marginal and the uniform distribution is measured. In what follows we measure the difference to equilibrium in term of separation discrepancy. The separation discrepancy between two probability measures $\mu$ and $\nu$ defined on the same measurable space
is given by
\bqn{sepdef}
\fs(\mu,\nu)&=&\esssup_{\nu} \big(1-\f{d\mu}{d\nu} \big)\eqn
where $d\mu/d\nu$ is the Radon-Nikodym derivative of the absolute continuous part  of $\mu$ with respect to $\nu$.

In a previous work \cite{Mag-Cou}, we describe the  behavior of $\fs(\cL(X^n(t),\mathcal{U}_{n}) $ where $M^n$ is the sphere or projective spaces (normalized in order to have the same diameter $\pi $). We obtain the following result,  Theorem 4 in \cite{Mag-Cou}:
\begin{align}
   &\lim_{c \to \infty}\limsup\limits_{n \to \infty}  \fs(\cL(X^n(\frac{a\ln(n)}{n} + \frac{c}{n}) ) ),\mathcal{U}_{n})  =0 ,&    \label{cut-pro}\\
  &\lim_{c \to \infty} \liminf\limits_{n \to \infty}  \fs(\cL(X^n (\fs(\cL(X^n(\frac{a\ln(n)}{n} - \frac{c}{n})),\mathcal{U}_{n})  = 1,& \notag \\
  \notag
  \end{align}

 where $\cL(X^n(t)) $ is the law of the Brownian motion at time $t$ and $ a = 1$ for the spheres case and $ a= 2$ for projective cases. This result prove that around the time $ \frac{a\ln(n)}{n} $ the separation discrepancy between $ X^n$ to the uniform distribution abruptly drop from the largest value $1$ to the smallest value $ 0$ during a window of size $ \frac1n. $ This phenomenon is the so called cutoff phenomenon. 
 
 \begin{defi}\label{def-cut}
The family of diffusion processes $(X^n)_{n\in\NN\setminus\{1\}}$ with invariant measure $ \mathcal{U}_n$ has a cutoff in separation with cutoff times $(a_n)_{n\in\NN\setminus\{1\}} $ if

\begin{itemize}
   \item $ \fo r>0,\qquad\lim_{n\ri\iy} \fs(\cL(X^n((1+r)a_n)),\mathcal{U}_n) =0 $
   \item  $ \fo r\in(0,1),\qquad \lim_{n\ri\iy} \fs(\cL(X^n((1-r)a_n)),\mathcal{U}_n)=1 $
\end{itemize}
\end{defi}  

 \begin{defi}[c3 and w3 in \cite{Chen-Coste}] \label{def2}
 The family of diffusion processes $(X^n)_{n\in\NN\setminus\{1\}}$ with invariant measure $ \mathcal{U}_n$ has an $(a_n,b_n)$ cutoff in separation with cutoff times $(a_n)$, 
 and window $(b_n)_{n}$ if 
   
 \begin{itemize}
   \item  $b_n =o(a_n) $
    \item $\lim\limits_{r\to \infty}\limsup \limits_{n\to\infty}\mathfrak{s}(\mathcal{L}(X^n({a_n+rb_n})),\mathcal{U}_n)=0$
    \item $\lim\limits_{r\to \infty}\liminf \limits_{n\to\infty}\mathfrak{s}(\mathcal{L}(X^n({a_n-rb_n})),\mathcal{U}_n)= 1.$
\end{itemize}

Moreover we say that the windows $(b_n)_{n}$ is strongly optimal if for all $r > 0 $
$$0 <  \liminf \limits_{n\to\infty}\mathfrak{s}(\mathcal{L}(X^n({a_n+rb_n})),\mathcal{U}_n) \le \limsup \limits_{n\to\infty}\mathfrak{s}(\mathcal{L}(X^n({a_n-rb_n})),\mathcal{U}_n) <1 .$$

 \end{defi}    
  The cutoff phenomenon was discovered by Diaconis and Shahshahani  \cite{DS} and Aldous and Diaconis \cite{AD} in the context of card shuffling. Afterward, the cutoff phenomenon has been proven for a large variety of finite Markov chains, see e.g.\ Diaconis \cite{MR1374011}, Diaconis and Fill \cite{MR1071805}, Levin, Peres and Wilmer \cite{MR2466937}. The cutoff phenomenon for Markov processes  on a continuous state space  have been proved, e.g. Chen and Saloff-Coste \cite{MR1306030},  \cite{Chen-Coste}  proved the cutoff phenomenon  in total variation distance for the Brownian motions on the spheres.
 
    Note that is shown by Hermon, Lacoin and  Peres \cite{zbMATH06618510} that total variation and separation cutoff are not equivalent and neither one implies the other. See also M\'eliot \cite{MR3201989} for cutoff phenomenon in total variation distance in  classical compact symmetric spaces. For cutoff is separation for rotationally symmetric compact manifolds see \cite{sphere} and \cite{rot}.\\

  The aim of this paper is to refines \eqref{cut-pro} by deriving the shape of the relaxation to equilibrium. 
 More precisely we are interested in proving the existence and  finding the limiting profile:
   
$$  \lim_{n \to \infty}  \fs(\cL(X^n(\frac{a\ln(n)}{n} + \frac{c}{n}) ) ),\mathcal{U}_{n}) \quad \text{, for all }  c \in \mathbb{R}. $$
We are also interested in the cutoff in separation together with the  profile for Brownian motion in the flat torus. Questions like these have been addressed in the literature e.g.  Diaconis, Graham,  and Morrison in \cite{DGM} compute the cutoff  profile  (in total variation) for the random walk in hypercube and Lacoin in \cite{Lacoin-pro} compute the cutoff  profile (in total variation) for the simple exclusion process on circle.

 The general strategy of the proof is to build a dual process of the Brownian motion, and  derive a sharp strong stationary time. This permit to translate the separation discrepancy between the law of the Brownian motion and it's invariant measure in term of the tail distribution function of the covering time of the duals process and then  analyze carefully  this distribution using again intertwining with radial Laplacian. Note that we do not use representation theory, only relation on Jacobi's polynomials.
 \\
 
 The paper is organized as follows:
 \\
 
 In section 2, we describe a rectangular dual process of the Brownian motion on the flat torus $ \mathbb{T}^n $. The sharp strong stationary time is them express as a maximum of $n$ independent and identically distributed  random variables $ \tau_i$. This random variables $ \tau_i$ have the law of the hitting time of a Bessel $3$ process. We recall some  known results on the tail distribution of the hitting time of the Bessel process,  with the aim of comparing them to the point of view of  section 2.3, using intertwining, which we believe to be new.  We show that the cutoff time is $  \frac{2\log(n)}{\pi^2} $ see Theorem \ref{cut-off-tore}. Using contour integral we compute the asymptotic of the tail distribution, and since the cutoff times goes to infinity we  obtain an asymptotic profile of Gumbel type see Theorem \ref{profile-tore}, that is natural as the law of a maximum of iid random variables.  In subsection 2.3, using intertwining we compute the law  of $ \tau_i$ using the Neumann spectrum of the Laplacian in $[-1,1]$, this alternative point of view will be applicable in Section 3.
 
 In Section 3 we computes the cutoff in separation  profile function for Brownian motion in $M^{n}$, where $M^{n}$ is a $n$ dimensional spheres $ \mathbb{S}^n$, real, complex, and quaternionic  projective space  resp. $ \mathbb{P}^n(\mathbb{R})$, $ \mathbb{P}^n(\mathbb{C})$ and $ \mathbb{P}^n(\mathbb{H})$.
 It was seen in \cite{zbMATH07470497} and \cite{arnaudon:hal-03037469},  that $X_n$ can be intertwined with a dual process  taking values in the closed balls of $M^{n}$ whose boundary evolution is a modified stochastic mean curvature flow, and whose ray generator is $L_n$ see section 3 for the definition. We use this dual process to construct a sharp and strong stationary time $\tau_{M^n} $. Using another intertwining  with the radial Laplacian see subsection 3.2, we obtain a correspondence between the spectrum of the radial Laplacian and the Green operator associated to  $ L_n$ see Proposition \ref{iso-spec}. This correspondence allows us to compute large moments of  $\tau_{M^n} $  see Theorem \ref{moment} and so the tail distribution   of $\tau_{M^n} $ see Theorem \ref{queue1}. This allows us to give a unified proof of the cutoff phenomenon for Brownian motion in $ M^n$ see Remark \ref{Remunif}. Contrary to the Torus case the cutoff times goes to $0$,  hence all the tail distribution function of $\tau_{M^n} $, is needed to compute the asymptotic cutoff profile, this is donne in subsection 3.4. see Theorem \ref{profile-S} and  Theorem \ref{profile-P}.  We obtain the main theorem of the paper :
   
 \begin{theo}
For $M^n =\mathbb{S}^n;   \mathbb{P}^n(\mathbb{C}), \mathbb{P}^n(\mathbb{H}) \text{ or }  \mathbb{P}^n(\mathbb{R}) $, the separation discrepancy  to equilibrium for the Brownian motion $X^n$  in $M^n$ has the following asymptotic profile: for all $c \in \mathbb{R} $   
$$\lim_{n \to \infty}  \fs(\cL(X_n( \frac{a\ln(n)}{n} + \frac{c}{n}) ) ,\cU_{M^n})  =    1- e^{- \frac{e^{-\frac{c}{a}}}{a}},$$
where $ a = 1$ for the spheres case and $ a= 2$ for projective cases.

The windows of the cutoff sequence  $(\frac{a\ln(n)}{n},\frac{1}{n})$ is strongly optimal in the sens of definition \ref{def2}. Moreover the above convergence is uniform on all compact.
\end{theo} 
 
\section{Cutoff in separation and profile for flat torus}\label{section2}

The goal of this section is to show a cutoff phenomenon  for the Brownian motion on the flat torus $ \mathbb{T}^n := (\mathbb{R}/ 2 \mathbb{Z})^n $, which we identify with the set $(-1,1]^n $. Let $x \in \mathbb{T}^n $ and $\mathbb{B}^n_t (x) $ the Brownian motion in  $ \mathbb{T}^n $ that  starts at $x = (x_1, ... , x_n)$. This Brownian motion could be describe as $ \mathbb{B}^n_t(x) = (B^1_t+x_1, ..., B^n+x_n) \quad \overline{mod}  \quad 2 \mathbb{Z} $, where $(B^1_t,..., B^n_t)$ is the usual Brownian motion in $\mathbb{R}^n$. Using the invariance by translation we will assume that  $ x=(0,...,0)$.
The invariant measure for $\mathbb{B}^n $ is the Lebesgue measure on the torus $\mu_{\mathbb{T}^n}(dx^1,...,dx^n))=\mu_{[-1,1]}(dx^1)\otimes...\otimes \mu_{[-1,1]}(dx^n)$. \\
It is well known (e.g. Pitman \cite{MR0375485}, or \cite{zbMATH07470497} ) that each $B^i$ can be intertwined with a process $(D^i(t))_{t\geq 0}= [-Bess^i_t, Bess^i_t]$ taking values in the closed interval of $ \mathbb{R} $, starting at $\{ 0\}$,  where $(Bess^i)_{i \in \{1..n \} }$ is a family of $n$ independent Bessel $3$ process.
 In  \cite{arnaudon:hal-03037469}, several couplings of  $B^i$ and $D^i$ were constructed, so that for any time $t\geq 0$, the conditional law of $X^i(t)$ knowing the trajectory $D^i({[0,t]})\df(D^i(s))_{s\in[0,t]}$ is the normalized uniform law over $D^i(t)$, which will be denoted $\Lambda(D^i(t),\cdot)$ in the sequel. Furthermore, $D^i$ is progressively measurable with respect to $B^i$, in the sense that for any $t\geq 0$, $D^i({[0,t]})$ depends on $B^i$ only through $B^i({[0,t]})$, one example of such coupling is the famous $ 2 Max - X   $ Pitman's Theorem.
 Let $ \tau_i = \inf \{t >0, Bess^i_t =1 \} $, and  $ \overline{\tau}_n := \max \{ \tau_i, i=1..n \}$.
 
  Using the independence of the coordinate, it appears that $\overline{\tau}_n $ is a strong stationary time for $\mathbb{B}^n$, meaning that $\overline{\tau}_n$ is a $\mathcal{F}^{\mathbb{B}^n}$ stopping time,  $\overline{\tau}_n$  and $\mathbb{B}^n_{\overline{\tau}_n}$ are independent, and $\mathbb{B}^n_{\overline{\tau}_n}$ is uniformly distributed over  $\mathbb{T}^n$. 
 Let $D(t) := D^1(t\wedge {\tau}_1) \times ... \times D^n(t \wedge {\tau}_n) \subset \mathbb{T}^n$, then $D(t)$ is a dual process of $\bb^n$ in the following sense :
 \bqn{cou1}
\fo t\geq 0,\qquad \cL(\bb^{n}_t\vert  D{[0,t]})&=&\Lambda_n(D(t),\cdot),\eqn
\par
and $(D(t)_{t\geq 0}$ can be constructed from $(\bb^n_t)_{t\geq 0}$ in an adapted way,
\bqn{cou2}
\fo t\geq 0,\qquad\cL(D{[0,t]}\vert \bb^n)&=&\cL(D{[0,t]}\vert \bb^n_{[0,t]}),\eqn\par
where the Markov kernel  $\Lambda_n$ is the restriction of the uniform measure $\mu_{\mathbb{T}^n} $, i.e. for all $ D$ domain or singleton  of $\mathbb{T}^n$
\bqn{Lambda}
  \fo A\in \cB(\mathbb{T}^n),\qquad \Lambda_n(D,A)&\df &
\lt\{
\begin{array}{lcl}  
&\frac{\mu_{\mathbb{T}^n} (A\cap D)}{\mu_{\mathbb{T}^n} (D)} &\hbox{, if $\mu_{\mathbb{T}^n}(D)>0$}\\
&\delta_x(A) & \hbox{, if $D=\{x\}$, with $x\in \mathbb{T}^n$}
\end{array}
\rt.
\eqn

As a consequence of such a dual and it's coupling, we have (see also Proposition \ref{sep}),
\bq
\fo t\geq 0,\qquad \fs(\cL(\bb^n(t)),\mu_{\mathbb{T}^n})&\leq & \PP[\overline{\tau}_n>t]\eq
where  the l.h.s.\ is the separation discrepancy between the law of $\bb^n(t) := \bb^n_t$ and the uniform distribution $\mu_{\mathbb{T}^n}$ over $\mathbb{T}^{n}$.

Note that for any $t\in [0,\overline{\tau}_n)$, the ``opposite pole'' $\overline{(1,...,1)} \in \mathbb{T}^n$ does not belong to the support of $\Lambda_n(D(t),\cdot)$.
It follows from an extension of Remark 2.39 of Diaconis and Fill \cite{MR1071805} that $\overline{\tau}_n$ is even a sharp  strong stationary time for $\bb^n$, meaning
that in fact
\bqn{sharp}
\fo t\geq 0,\qquad \fs(\cL(\bb^n(t)),\mu_{\mathbb{T}^n})&=& \PP[\overline{\tau}_n>t] .\eqn
\par
Thus the understanding of the convergence in separation of $\bb^n$ toward $\mu_{\mathbb{T}^n} $ amounts to understanding the  tail distribution of $\overline{\tau}_n$.

\subsection{Hitting time of Bessel 3 process}

Let $X_t(x) $ be the usual Bessel $3$ process that starts at $x \in \mathbb{R}^{+}$, it is the $L = \frac{1}{2} \partial^2_r + \frac{1}{r} \partial_r $ diffusion, and $0$ is an entrance boundary for  the process $ X$.
Let $\tau := \inf \{t >0, X_t(0)= 1 \} $, in this section we will be interested in the two sides estimate of the tail distribution of $ \mathbb{P} [ \tau >t ] $. We recall well known results,  with the aim of comparing them to the following section 2.3.
 The following  well know identity in law (e.g. (4.21)  in \cite{BYP}),  is also a consequence of \eqref{pilap} below.
$$ \tau  \eql \sum_{m=1}^{ \infty} \epsilon_m(\frac{j^2_{\frac12,m}}{2} ) , $$
where $  \epsilon_m (\frac{j^2_{\frac12,m}}{2} )$ are independent standard exponential variables with parameter $\frac{j^2_{\frac12,m}}{2} , $ 
and  $j_{\frac12,m} $ is the zero of the Bessel function $ J_{\frac12} $ of the first kind
 
 $$J_{\frac12}(x) =  \big( \frac{x}{2} \big)^{\frac12} \sum_{m=1}^{\infty} \frac{(-1)^m}{m! \Gamma(m+1+ \frac12)} \big( \frac{x}{2} \big)^{2m} = \sqrt{\frac{2}{\pi}} \frac{\sin(x)}{\sqrt{x}},$$
for the last equality use that $\Gamma(x +1)= x \Gamma(x) .$

 So $J_{\frac12,1} = \pi $, and 
we obtain the following  lower bound :
\bqn{lower}
 \exp (-\frac{\pi^2}{2} t) = \mathbb{P} [ \epsilon_1(\frac{j_{\frac12,1}}{2} )  >t ]  \le  \mathbb{P} [ \tau >t ].
\eqn

For an upper bound, let $L := \frac12 (\partial^2_r + \frac{2}{r} \partial_r) $ be the generator of  $ X $, we will compute all moments of $\tau$, and their equivalent.

\begin{pro}\label{Poisson}
Given $ g \in C_b([0, 1])$, the bounded solution $\phi \in C^2_b$ of the Poisson equation
\bq
 \lt\{\begin{array}{rcl}
  L \phi &=& -g  \\ 
  \phi (1) &=& 0 \\
\end{array}\rt.\eq
is given by:
 \bqn{phinr}
 \fo r\in[0,1],\qquad
  \phi (r) &=&  \int_r^1  \frac{1}{t^2}  \lt(\int_0^t  2s^2 g(s) ds \rt) dt.  
\eqn
    So the Green operator $G$ associated to $L$ is given by 
    \bq 
    \fo g\in C_b([0, 1]),\,\fo r\in[0,1],\qquad
    G[g](r) = 2\int_r^1  \frac{1}{t^2}  \lt(\int_0^t  s^2 g(s) ds \rt) dt= 2\wi{G}[g](r)
    \eq
 \end{pro} 
 
 \begin{proof}
 It is a straight forward computation.
\end{proof}  

Let $u_{0}\df \un $, the constant function taking the value 1 on $[0,1]$, and consider the following sequence  $(u_{k})_{k\in\NN} $, defined inductively by bounded solution of

 \bqn{def-u}
\fo k\in\NN,\qquad \lt\{\begin{array}{rcl}
 L u_{k} &=& -k   u_{k-1}\\
    u_{k} (1) &=& 0 .\\
\end{array}\rt.\eqn

 We have for all $k \in \mathbb{Z}_+$,  $$ \frac{u_{k}}{k !} =  G^{\circ k}[\un]\df G[G[\cdots [G [\un]]\cdots]]= 2^k \wi{G}^{\circ k}[\un]. $$

\begin{pro}\label{mean-eq}
For any $k\in\ZZ^+$,
 \bq  \EE[\tau^{k}] &=& u_{k}(0)\, = k! G^{\circ k}[\un] (0) = k!2^k \wi{G}^{\circ k}[\un](0) \\
 &=& k!2^k \big( \frac{2}{\pi^{2k}} - \frac{2}{(2\pi)^{2k}} + o (\frac{1}{(2\pi)^{2k}}) \big) \\
 &=& k!2^k \frac{2}{\pi^{2k}}  ( 1- 2^{1-2k} )\zeta(2k)\\
 &  \sim_{k}& k! 2 (\frac{2}{\pi^2})^k ,\eq
where $\zeta$ is the Riemann zeta function.
\end{pro}
\begin{proof}
Suppose by induction that $ u_{k}(x)  =  \EE_x [\tau^k] .$ This is clearly satisfied for $k=0$.
Let $ X_0 = x \in [0,1]$, using It\^o's formula, we have for all $ 0\le t \leq  \tau$, 
$$ u_{k+1}(X_t) - u_{k+1}(x) = - (k+1)\int_0^t u_{k} (X_s)  ds + M_t $$
where $ (M_t)_{t\in[0,\tau]}$ is a martingale. Consider this equality with $t=\tau$, take expectation and use the Markov property to get
 \begin{align*}
u_{k+1}(x)  &=  (k+1)\EE_x \lt[\int_0^{\tau} u_{k}(X_s) ds\rt] \\
             &= (k+1)\EE_x \lt[\int_0^{\tau} \EE_{X_s}[\tau^k] ds\rt] \\
             &= (k+1)  \EE_x \lt[\int_0^{\tau} (\tau - s)^k ds\rt] = \EE_x[\tau^{k+1}] .\\
\end{align*}  

Note that $$\wi{G}[\un](r)= \int_r^1  \frac{1}{t^2}  \lt(\int_0^t  s^2 \un(s) ds \rt) = \frac16 (1-r^2) $$
and for all $m \in \ZZ^+$ 
$$\wi{G}[x\mapsto x^m](r)= \int_r^1  \frac{1}{t^2}  \lt(\int_0^t  s^2 s^m ds \rt) = \frac{1}{(m+3)(m+2)}(1-r^{m+2}). $$
Let $b_k = \wi{G}^{\circ k}[\un](0) ,$ using the above formula we obtain the following recursive formula $b_0=1$ and for $ 1 \le k $ :
$$b_{k} = -\big( \sum_{m=1}^ {k} b_{k-m} \frac{(-1)^m}{(2m+1)!}   \big) ,$$
i.e. for $ 1 \le k,$ $$ b_0 = 1, \quad and \quad  \sum_{m=0}^ {k} b_{k-m} \frac{(-1)^m}{(2m+1)!}  =0 .$$
We  identify this recursive equation as follows  for $z \in \mathbb{C}$
\bqn{eq-b}
\frac{1}{  \sum_{n=0}^{\infty}  \frac{(-1)^n}{(2n+1)!} z^n} = \sum_{n=0}^{\infty} b_n z^n,
\eqn
namely
\bqn{cauchy}
\frac{z}{\sin{z}} = \frac{1}{  \sum_{n=0}^{\infty}  \frac{(-1)^n}{(2n+1)!} z^{2n}} = \sum_{n=0}^{\infty} b_n z^{2n}.
\eqn

Since $\frac{1}{\sin{z}} $ is a meromorphic function and the five first poles are $ \{ -2\pi , -\pi , 0 , \pi , 2\pi \} $, let $\pi < R <2\pi $, and for $\theta \in [0, 2 \pi]$ and define the following contours for  $ 0 <\eta $ small enough
\bq
 \gamma_1(\theta) = \pi + \eta (\cos(-\theta)+ i \sin(-\theta)) \\
 \gamma_2(\theta) = -\pi + \eta (\cos(-\theta) + i \sin(-\theta)) \\
 \gamma_3(\theta) = R(\cos(\theta) + i \sin(\theta)). \\
 \eq
Note that $\gamma_3$ is anticlockwise  whereas $\gamma_1,\gamma_2 $ are clockwise, using Cauchy formula, for all $n \ge 1$:
\begin{align*}
b_n& = \frac{1}{2i\pi} \int_{\gamma_1 \cup \gamma_2 \cup \gamma_3 } \frac{z}{\sin(z) z^{2n+1}} dz \\
&= \frac{1}{2i\pi} \int_{\gamma_1 \cup \gamma_2 } \frac{z}{\sin(z) z^{2n+1}} dz + O_n(\frac{1}{R^{2n-1}})\\
&= -res(\frac{z}{\sin(z) z^{2n+1}}, \pi ) - res (\frac{z}{\sin(z) z^{2n+1}}, -\pi) + O_n(\frac{1}{R^{2n-1}})\\
&= -res(\frac{-\pi}{(z-\pi) \pi^{2n+1}}, \pi ) - res (\frac{\pi}{(z+\pi) (-\pi)^{2n+1}}, -\pi) + O_n(\frac{1}{R^{2n-1}})\\
&=\frac{2}{\pi^{2n}} + O_n(\frac{1}{R^{2n-1}})\\
&\sim_n  \frac{2}{\pi^{2n}},\\
\end{align*}
where in the forth line we have used that $ \sin(z) \sim_{z\sim \pi} -(z-\pi) $ and $ \sin(z) \sim_{z\sim -\pi} -(z+\pi) .$
Adding  two clockwise contours  around the second pole of $ -2\pi $ and $2\pi$ and take as in the above computation $ 2\pi< R < 3\pi$, and the equivalent of $ \sin $ at $ -2\pi $ and $2\pi$
, we obtain $$ b_n = \frac{2}{\pi^{2n}} - \frac{2}{(2\pi)^{2n}} + o (\frac{1}{(2\pi)^{2n}}), \forall n \ge 1 \quad and \quad b_0 = 1. $$
With similar argument we have in fact :
$$b_n  = \frac{2}{\pi^{2n}} \sum_{k=1}^{\infty } \frac{(-1)^{k+1}}{k^{2n}}  =  \frac{2}{\pi^{2n}}  ( 1- 2^{1-2n} )\zeta(2n),$$
where $\zeta$ is the Riemann zeta function.
\end{proof}

\begin{pro}\label{two-side}
For all $ 0 \le  \lambda  < \frac{\pi^2}{2}$ and $ t>0$:
 
 $$ \exp (-\frac{\pi^2}{2} t)  \le  \mathbb{P} [ \tau >t ]  \le  \phi_{\tau}(\lambda)  \exp (-\lambda t)  ,$$
 where  $ \phi_{\tau}(\lambda) = \EE [\exp{(\lambda \tau})] $ is the moment generating function of $\tau$.
\end{pro}
\begin{proof}

 We have for $ 0 \le \lambda $,  $  \phi_{\tau}(\lambda) = \sum_{k=0}^{\infty} \lambda^k \frac{\EE[\tau^k]}{k!}, $
 using Proposition \ref{mean-eq} we write 
 \bq   \phi_{\tau}(\lambda) =  1+  2 \sum_{k=1}^{\infty}  \big( \frac{2 \lambda}{\pi^2} \big)^k   \big( 1 - \frac{1}{(2)^{2k}} + o (\frac{1}{(2)^{2k}}) \big)    \eq
 and $  \phi_{\tau} (\lambda)$  is convergent for  $ 0 \le  \lambda  < \frac{\pi^2}{2}.$  
 Hence we have the following upper bound, using Markov inequality, for all $ 0 \le  \lambda  < \frac{\pi^2}{2}$ and $ t>0$:
 
 $$  \mathbb{P} [ \tau >t ] = \mathbb{P} [ \lambda \tau >  \lambda t ]   \le  \phi_{\tau}(\lambda)  \exp (-\lambda t)  $$  
 \end{proof}

\begin{rem}
Using the moment generating function of $\tau$, Proposition \ref{mean-eq}, \eqref{cauchy}, and the following Weierstrass decomposition :
$$ \sin(\pi z)= \pi z \prod_{n=1}^{\infty} \Big( 1 - (\frac{z}{n})^2 \Big),$$
we get 
\bqn{pilap}
 \phi_{\tau}(\lambda) &=& \EE [\exp{(\lambda \tau})] = \sum_{k=0}^{\infty}  (2\lambda ) ^k \wi{G}^{\circ k}[\un](0) \notag \\ 
 &=&  \sum_{k=0}^{\infty}  (2\lambda ) ^k b_k = \frac{\sqrt{2\lambda}}{\sin{\sqrt{2\lambda}}} = \prod_{n=1}^{\infty} \frac{1}{ 1 - (\frac{2\lambda}{\pi^2 n^2})}.\notag \\ 
 \eqn
 After identification we get 
 $$ \tau  \eql \sum_{m=1}^{ \infty} \epsilon_m(\frac{\pi^2 m^2}{2} ) , $$
where $  \epsilon_m ( \frac{\pi^2 m^2}{2} )$ are independent standard exponential variables with parameter $\frac{\pi^2 m^2}{2} $.
 
\end{rem} 
 
\begin{rem}

Using the above Proposition \ref{mean-eq} we get  $\EE[\tau^{k}] = k!2^k \big( \frac{2}{\pi^{2k}} - \frac{2}{(2\pi)^{2k}} + o (\frac{1}{(2\pi)^{2k}}) \big) $, and Stirling formula,  we get $$ (\EE[\tau^{k}])^{\frac1k} \sim \frac{2}{\pi^2}  (k!)^{\frac1k} \sim \frac{2k}{\pi^2 e} .$$
\end{rem}
We get the following cutoff phenomenon, in the sens of definition \ref{def-cut}, for Brownian motion on the torus.


\begin{theo}There is a cutoff in separation  for the Brownian motion on the flat torus $ \mathbb{T}^n := (\mathbb{R}/ 2 \mathbb{Z})^n $, with cutoff time $  \frac{2\log(n)}{\pi^2} $.\label{cut-off-tore}
\end{theo}
\begin{proof}
Using the upper bound in Proposition \ref{two-side} we get for all $ 0 \le t $ and $ 0 \le \lambda <\frac{\pi^2}{2}$,
we have :

\bq
\mathbb{P}[ \overline{\tau}_n \le t ]& =& (\mathbb{P}[ {\tau} \le t ])^n 
        \ge   \big( 1- \phi_{\tau}(\lambda)  \exp (-\lambda t)   \big) ^n .\\
\eq
  
 For all $ r > 0$, let $ \lambda_r = \frac{\pi^2}{2} \big( \frac{1}{1+\frac12 r} \big) < \frac{\pi^2}{2}$ and $ r'= \frac{r}{2(1+ \frac12 r)} > 0$ we have $ \frac{2}{\pi^2} (1+r)= \frac{1}{\lambda_r} (1+ r' )$ and
 
  \bq
\mathbb{P}[ \overline{\tau}_n > \frac{2\log(n)}{\pi^2} (1+r) ] &= & \mathbb{P}[ \overline{\tau}_n > \frac{\log(n)}{\lambda_r} (1+r') ]\\
& =& (1 - \mathbb{P}[ {\overline{\tau}_n} \le \frac{\log(n)}{\lambda_r} (1+r')  ]  ) \\
        &\le &  1-  \big( 1- \phi_{\tau}(\lambda_r)  \exp^{-\log(n) (1+r') }   \big) ^n \\
        &=& 1 -(1- \frac{\phi_{\tau}(\lambda_r)}{n^{1+r'} })^n \to 0 \\
\eq

Using the lower  bound in Proposition \ref{two-side} we get for $ 0 <r <1$
\bq
\mathbb{P}[ \overline{\tau}_n < \frac{2\log(n)}{\pi^2} (1-r)  ]& =& (\mathbb{P}[ {\tau} < \frac{2\log(n)}{\pi^2} (1-r)  ])^n \\
&=& ( 1 - \mathbb{P}[ {\tau} \ge \frac{2\log(n)}{\pi^2} (1-r)  ])^n \\
&\le & ( 1 -  \exp^{ -\log(n) (1-r)}  )^n \\
&=& (1- \frac{1}{n^{1-r}})^n \to 0 \\
\eq
Using the two computations above and \eqref{sharp} we get the conclusion.
\end{proof}

\subsection{The profile function for torus} \label{complexe}
Let $ \phi(\lambda) = \EE [\exp{(-\lambda \tau})] $ be the Laplace transform of $\tau$, using \eqref{eq-b} we get the well known formula :
\begin{align}
 \phi(\lambda) &= \sum_{n=0}^{\infty} (- \lambda)^n \frac{\EE[\tau^n]}{n!} = \sum_{n=0}^{\infty} (- 2\lambda)^n b_n \notag \\
 & = \frac{1}{  \sum_{n=0}^{\infty}  \frac{1}{(2n+1)!} (2\lambda)^{n}} = \frac{\sqrt{2\lambda}}{ \sinh(\sqrt{2\lambda} )} \label{sinh}.
 \end{align}
  
Since the poles of $ \phi $ are $ \{ -\frac{\pi^2 n^2}{2},\quad n \in \mathbb{Z}^*    \} $.
   Using formula (Theorem 8.1 \cite{Widder}) for the inverse of the Laplace transform we have $ \forall x \ge 0 $   

 $$ \mathbb{P} [\tau \le x ]  =\lim_{R \to \infty } \frac{1}{2i\pi } \int_{1- i R}^{1 + i R} e^{z x} \frac{\phi (z)}{z} dz.  $$ 
 
 Let $D_R = D(0, R) \cap \{ Re(z) \le 1 \} $, and $ \mathcal{P} := \{ -\frac{\pi^2 n^2}{2},\quad n \in \mathbb{Z}    \}$  be the set of poles of $ \frac{\phi (z)}{z}$ ,  we have by Cauchy formula
 \begin{align}
 \frac{1}{2i\pi } \int_{\partial D_R} e^{z x} \frac{\phi (z)}{z} dz &= \sum_{\alpha \in D_R \cap  \mathcal{P} }   Res (  e^{z x} \frac{\phi (z)}{z} ; \alpha   )  \notag \\ 
 &=   \sum_{\alpha \in D_R \cap  \mathcal{P} }  e^{ \alpha x} Res (   \frac{\phi (z)}{z} ; \alpha   )   \notag \\ 
 &=  \sum_{n = 1}^{\lfloor \sqrt{ \frac{2R}{\pi^2} } \rfloor  } e^{ -\frac{\pi^2n^2}{2}  x} Res (   \frac{\phi (z)}{z} ;  -\frac{\pi^2n^2}{2}    ) +1 \notag \\
 &=  \sum_{n = 1}^{\lfloor \sqrt{ \frac{2R)}{\pi^2} } \rfloor  } 2(-1)^{n} e^{ -\frac{\pi^2n^2}{2}  x} +1 , \label{res} \\
 \notag
 \end{align}
 where we use in the last equality that $ Res (   \phi (z);  -\frac{\pi^2n^2}{2}    ) = \pi^2 n^2 (-1)^{n+1}$ for  $n>0$ and $ Res (   \frac{\phi (z)}{z} ;  0   ) = 1$.
 
 Let $C_R = \mathcal{C}(0, R) \cap \{ Re(z) \le 1 \} $ the left boundary of $D_R$,
 and let $z := Re^{i \theta} $ for $\theta \in [\arccos (\frac{1}{R}), 2\pi - \arccos (\frac{1}{R})  ] $,  since   $ \vert \phi (z) \vert = \frac{\sqrt{2 \vert z \vert }}{\sqrt{ \sinh^2( \sqrt{\vert z \vert + Re(z)})  + \sin^2(\sqrt{\vert z \vert  - Re(z)}  ) } } =  \vert \phi (\bar{z}) \vert  $, we have for all $ x>0$
 $$ \vert \int_{C_R} e^{z x} \frac{\phi (z)}{z} dz \vert \le 2 \sqrt{2R} \int_{ \arccos (\frac{1}{R}) }^{\pi}    \frac{e^{ Rx \cos(\theta)}}{\sqrt{ \sinh^2( \sqrt{2R} \cos(\theta / 2)  + \sin^2( \sqrt{2R} \sin(\theta / 2 )}}    d\theta .$$  Chose    $R$ such that $ \sqrt{2R}  = 2m\pi+ \frac{\pi}{2}  $, and  cut the integral in the left hand side of the above equation into three parts
 $$ I_1(R) := \int_{ \arccos (\frac{1}{R}) }^{\pi /2 }    \frac{e^{ Rx \cos(\theta)}}{\sqrt{ \sinh^2( \sqrt{2R} \cos(\theta / 2))  + \sin^2( \sqrt{2R} \sin(\theta / 2 ))}}    d\theta, $$
 $$ I_2(R) := \int_{\pi /2}^{ \pi - \frac{1}{\sqrt{2R}} }    \frac{e^{ Rx \cos(\theta)}}{\sqrt{ \sinh^2( \sqrt{2R} \cos(\theta / 2))  + \sin^2( \sqrt{2R} \sin(\theta / 2 ))}}    d\theta, $$
 $$ I_3(R) := \int_{\pi - \frac{1}{\sqrt{2R}} }^{ \pi }    \frac{e^{ Rx \cos(\theta)}}{\sqrt{ \sinh^2( \sqrt{2R} \cos(\theta / 2))  + \sin^2( \sqrt{2R} \sin(\theta / 2 ))}}    d\theta.$$
 We have :
 $$I_1(R) \le \frac{1}{ \sinh( \sqrt{R}) } \int_{ \arccos (\frac{1}{R}) }^{\pi /2 } {e^{ Rx \cos(\theta)}} d\theta \le  \frac{\pi}{ 2\sinh( \sqrt{R}) } e^x,$$
 also since for $ \theta \in [\pi - \frac{1}{\sqrt{2R}}, \pi  ] $, $\frac12 \le \vert \sin( \sqrt{2R} \sin(\theta /2 )\vert  $ we have for large $R$ 
 $$ I_3(R) \le 2 \int_{\pi - \frac{1}{\sqrt{2R}} }^{ \pi } e^{ Rx \cos(\theta)} d\theta  \le 2 e^{-Rx/\sqrt{2}}.$$
 For the last one, we have
 \begin{align*}
 I_2(R) & \le \frac{1}{\sinh( \sqrt{2R}\sin( \frac{1}{2\sqrt{2R}} ))} \int_{\pi /2}^{ \pi - \frac{1}{\sqrt{2R}} }e^{ Rx \cos(\theta)}  d\theta  \\
 & \le \frac{1}{\sinh( \sqrt{2R}\sin( \frac{1}{2\sqrt{2R}} ))}  \int_0^1 \frac{e^{-(Rx) y}}{\sqrt{1-y^2}} dy = O(1) \int_0^1 \frac{e^{-(Rx) y}}{\sqrt{1-y^2}} dy ,\\
  \end{align*}
  and 
  \begin{align*}
  \int_0^1 \frac{e^{-(Rx) y}}{\sqrt{1-y^2}} dy &= \int_0^{1/\sqrt{R} } \frac{e^{-(Rx) y}}{\sqrt{1-y^2}} dy + \int_{1/\sqrt{R} }^1 \frac{e^{-(Rx) y}}{\sqrt{1-y^2}} dy \\
  &\le \frac{1}{Rx\sqrt{1-1/R}} + e^{-\sqrt{R}x} \int_{0}^1 \frac{1}{\sqrt{1-y^2}} dy =O(\frac{1}{R}).
  \end{align*}
  Hence $ I_1(R) + I_2(R) + I_1(R) = O(\frac{1}{R})$ and for all $ x>0$ we have
 \bqn{bord}  \lim_{R \to \infty } \frac{1}{2i\pi }  \int_{C_R} e^{z x} \frac{\phi (z)}{z} dz  = 0  .\eqn
Hence passing to the limit in \eqref{res} and using \eqref{bord} we get the following well know formula for all $ x>0$
\bqn{laplace} 
 \mathbb{P} [\tau \ge x ]  = \sum_{n = 1}^{\infty} 2(-1)^{n+1} e^{ -\frac{\pi^2n^2}{2}  x} = 2 e^{ -\frac{\pi^2}{2}  x} + o(e^{ -\frac{\pi^2}{2}  x} ). 
\eqn
 
 The aim of this section is to complete Theorem \ref{cut-off-tore} by the shape of the relaxation in separation  to equilibrium. In particular, we are interested in proving existence of optimal window and  computing the limiting profile.
 
 \begin{theo} \label{profile-tore}The separation to equilibrium for the Brownian motion on the flat torus $ \mathbb{T}^n := (\mathbb{R}/ 2 \mathbb{Z})^n $, has the following asymptotic profile,  for any $ c \in \mathbb{R} $ we have : 
   $$ \lim_{n\ri\iy} \fs(\bb^n( \frac{2\log(n)}{\pi^2} + c ),\mu_{\mathbb{T}^n}) = 1-e^{-2 e^{-\frac{\pi^2}{2}c } }, $$
 the above convergence is uniform on all compact  and the windows of the cutoff sequence  $(\frac{2\log(n)}{\pi^2},1)$ is strongly optimal in the sens of definition \ref{def2}. 
\end{theo}
\begin{proof}
Let $ c \in \mathbb{R} $.
Using \eqref{sharp} and \eqref{laplace} we have :
\begin{align*}
 \fs(\cL(\bb^n(\frac{2\log(n)}{\pi^2} + c ))),\mu_{\mathbb{T}^n}) =& \PP[\overline{\tau}_n>\frac{2\log(n)}{\pi^2} + c )] \\
 =& 1 - \mathbb{P}[ \overline{\tau}_n \le \frac{2\log(n)}{\pi^2} + c )]   \\
 =& 1 - (\mathbb{P}[ \tau \le \frac{2\log(n)}{\pi^2} + c )])^n   \\
  =& 1 - (1 -  2 e^{ -\frac{\pi^2}{2}  (\frac{2\log(n)}{\pi^2} + c ) } + o(e^{ -\frac{\pi^2}{2}  (\frac{2\log(n)}{\pi^2} + c ) } ) )^n   \\
  =& 1 - (1 -  2 \frac{e^{ -\frac{\pi^2}{2} c  }}{n} + o(\frac1n ) )^n  \to   1-e^{-2 e^{-\frac{\pi^2}{2}c } }, \\
 \end{align*} 
 and  the  convergence is  uniform on all compact.
\end{proof}   

\begin{rem}
The above computation shows that we have the following convergence in law :
\begin{align*}
  \overline{\tau}_n - \frac{2 \ln(n)}{\pi^2} \underset{n \to +\infty}{\overset{\mathcal{L}}{\longrightarrow}} Gumbel \big( \frac{2 \ln(2)}{\pi^2}, \frac{2}{\pi^2} \big)
\end{align*}
where $Gumbel $ is the  Gumbel distribution.
\end{rem}

\subsection{An alternative point of view for hitting time using intertwining}

The goal of this section is to give an alternative approach of the precedent sections that will be interesting for finding the cutoff profile for the Brownian motion on spheres and projective spaces in the next section. The precedent approach need at some place an explicit computation of the moment generated function and it's factorizations \eqref{pilap} or  an explicit  Laplace transform    for the computation of the poles \eqref{sinh}, and in the forthcoming section the Laplace transform will be not so explicit.

Let $ B$ be a one dimensional Brownian motion that start at $0$, i.e. a $ \Delta :=  \frac12 \frac{\partial^2}{\partial^2x} $ diffusion.  
Let $X $ be the usual Bessel $3$ process that starts at $0$, i.e. the $L = \frac{1}{2} \partial^2_r + \frac{1}{r} \partial_r $ diffusion. It is well known that e.g. Theorem 3 in \cite{zbMATH07470497}  that the generator of  $B$ can be intertwined with the generator of $X $ in the following sens :

   \bqn{entrela}
    L \Lambda = \Lambda  \Delta   
    \eqn 
where for $ r\ge 0$, $ \Lambda (f) (r) = \frac{\int_{-r}^{r} f(x) dx }{2r} $, note that in this case \eqref{entrela} also follows by direct computation.
Let $\mathbb{Z}_+ := \{ 0,1,2,3...  \}  $ and $ \mathbb{Z}^+ := \{1,2,3,...  \}  $.
Let $\varphi_n(x) = \cos(n\pi x) $ for $ n \in \mathbb{Z}_+ $, be the Neumann eigenfunctions of   $\Delta $ in  $C^{2}[-1,1] $ associated to the eigenvalues $ \lambda_n = - \frac{(n\pi)^2}{2} $. 
The eigenfunctions $\{ \varphi_0/ \sqrt{2}, \, \varphi_n \quad n \in \mathbb{Z}^+ \}$  form  an orthonormal  system of functions in $L^2([-1;1], dx) $.
Using \eqref{entrela} we get that $\Lambda \varphi_n $ for $n \in \mathbb{Z}^+ $ is an eigenfunction of $L $ associated to $ \lambda_n$, with boundary conditions $f(0) = 1 $ and $f(1)=0$.
\begin{pro}
The family  $ \{ n\pi \sqrt{2} \Lambda \varphi_n , \quad    n \in \mathbb{Z}^+\}$  is an orthonormal complete  system of functions in $L^2([0;1], x^2dx) $ where $x^2dx  $ is the invariant measure of $ L$. 
 Also  $ \{ \wi{\Lambda \varphi_n} :=  n\pi \sqrt{2} \Lambda \varphi_n , n \in \mathbb{Z}^+  \} $ are an Hilbert basis of  eigenfunction of $ G $ associated respectively to the eigenvalue $ \frac{-1}{\lambda_n}$.

\end{pro}
\begin{proof}
For $ n \in \mathbb{Z}^+ $, $\Lambda \varphi_n (r) = \frac{\sin(n\pi r) }{n\pi r }$, so for $n,m \in \mathbb{Z}^+ $
\begin{align*}
\langle \Lambda \varphi_n , \Lambda \varphi_m \rangle_{L^2([0;1], x^2dx)} &= \int_{0}^{1}   \frac{\sin(n\pi r) }{n\pi r }  \frac{\sin(m\pi r) }{m\pi r }r^2 dr \\
&=  \frac{1}{2nm \pi^2} \delta_{nm},
\end{align*}
hence $ \{ n\pi \sqrt{2} \Lambda \varphi_n , \quad    n \in \mathbb{Z}^+\}$  is an orthonormal system of functions.
For the completness, let $f \in L^2([0;1], x^2dx )$ such that for all $n \in \mathbb{Z}^+ $,    $ \langle f , \Lambda \varphi_n \rangle_{L^2([0;1], x^2dx)} =0$ then 
$$ \langle x f , \sin(n\pi x) \rangle_{L^2([0;1], dx)} =0  \quad \forall n \in \mathbb{Z}^+  .$$ 
Since $ \sin(n\pi x)$ for $n \in \mathbb{Z}^+ $ is a complete system of functions in ${L^2([0;1], dx)}$, we get that $f =0 $.

Since by \eqref{entrela} we have $L  \wi{\Lambda \varphi_n} = \lambda_n \wi{\Lambda \varphi_n}, $ so for $n \in \mathbb{Z}^+  $
 $$ L  \frac{- \wi{\Lambda \varphi_n}}{\lambda_n} =   -\wi{\Lambda \varphi_n}  ,$$ and by definition of $G$ in Proposition \ref{Poisson} we have 
 $$ G( \wi{\Lambda \varphi_n} ) = \frac{-1}{\lambda_n} \wi{\Lambda \varphi_n} .$$
 For $f \in L^2([0;1], x^2dx )$ we have $$ G^{\circ k}(f)(r) = \sum_{n=1}^{\infty} \Big( \frac{-1}{\lambda_n}  \Big)^k \wi{\Lambda \varphi_n} (r) \langle f , \wi{\Lambda \varphi_n }\rangle_{L^2([0;1], x^2dx)} .$$ 

\end{proof}   
\begin{cor} The hitting time of the Bessel $3$ process $\tau := \inf \{t >0, X_t(0)= 1 \} $ have the following density :
 
 $$ f_{\tau}(x) =    \sum_{n=1}^{\infty} (-1)^{n+1} (n\pi)^2 e^{- \frac{(n\pi)^2x}{2}} \un_{[0, \infty[}(x).$$
\end{cor}
\begin{proof}
 The Laplace transform of $\tau$,
 \begin{align*}
 \phi(\lambda) &= \sum_{k=0}^{\infty} (- \lambda)^k \frac{\EE[\tau^k]}{k!} = \sum_{k=0}^{\infty} (- \lambda)^k  G^{\circ k}(\un)(0) \\
 &=1 +  \sum_{k=1}^{\infty} (- \lambda)^k  \sum_{n=1}^{\infty} \Big( \frac{-1}{\lambda_n}  \Big)^k \wi{\Lambda \varphi_n} (0) \langle \un , \wi{\Lambda \varphi_n} \rangle_{L^2([0;1], x^2dx)} \\ 
   &=1 +  \sum_{k=1}^{\infty} (- \lambda)^k  \sum_{n=1}^{\infty} \Big(\frac{-1}{\lambda_n}  \Big)^k 2 (n\pi)^2 \langle \un , {\Lambda \varphi_n} \rangle_{L^2([0;1], x^2dx)} ,\\
    &= 1+ 2  \sum_{n=1}^{\infty} (-1)^{n+1} \frac{\lambda}{\lambda_n} \frac{1}{ 1- \frac{\lambda }{\lambda_n}     } = 1- \lambda 2  \sum_{n=1}^{\infty} (-1)^{n+1} \frac{ 1}{ \lambda - \lambda_n     } ,\\
    &=1 - \lambda 2  \sum_{n=1}^{\infty} (-1)^{n+1} \mathcal{L} (x \mapsto  e^{\lambda_n x} )( \lambda) \\
    &=1- \lambda \mathcal{L} ( x \mapsto 2  \sum_{n=1}^{\infty} (-1)^{n+1}   e^{\lambda_n x} )( \lambda)\\
 \end{align*}
  where in the fourth  equality we use that $ \langle \un , \Lambda \varphi_n \rangle_{L^2([0;1], x^2dx)} = \frac{(-1)^{n+1} }{(n\pi)^2 } $ and Fubini Tonelly Theorem (for $ \vert \lambda \vert < 1$ and holomorphic prolongation), and where $\mathcal{L}$ is the Laplace transform.
  From e.g. Lemma \ref{lemme-Lap} below, we get :
  \begin{align*}
   \mathbb{P} [\tau \ge x ]  = \sum_{n = 1}^{\infty} 2(-1)^{n+1} e^{ \lambda_n x},
   \end{align*}
   and the result after differentiation for $x>0 $. 
\end{proof}

\begin{rem}
The above result gives  another proof of \eqref{laplace} that only use the Neumann spectrum. 

 The hitting time of the Bessel $3$ process that start at $ 0\le y < 1 $,  $\tau_y := \inf \{t >0, X_t(y)= 1 \} $ have the following density :
 $$ f_{\tau_y}(x) =    \sum_{n=1}^{\infty} (-1)^{n+1}   \frac{\sin(n\pi y) }{ y }   (n\pi) e^{- \frac{(n\pi)^2x}{2}} \un_{[0, \infty[}(x).$$
 \end{rem}

\section{The cutoff profile in separation for Spheres and projective spaces}
\subsection{Intertwining}
In this section we will be specially interested in computing the cutoff in separation  profile function for Brownian motion in $n$ dimensional spheres $ \mathbb{S}^n$, real, complex, and quaternionic  projective space  resp. $ \mathbb{P}^n(\mathbb{R})$, $ \mathbb{P}^n(\mathbb{C})$ and $ \mathbb{P}^n(\mathbb{H})$. Where the metric come from writing the manifold as homogeneous space see \cite{Jaming},  all of them are normalized in order to have the same diameter $\pi $, and $n$ will be the real dimension. We will be particularly interested in this type of manifold since their are compact two-point homogeneous spaces  and hence harmonics manifolds in the sens that all sphere in this spaces have constant mean curvature. A compact Riemannian manifold $(M,g)$ with distance $d_g(.,.) $ is  two-point homogeneous if for all $x_1,x_2,y_1,y_2 \in M $, such that  $d_g( x_1,x_2)=d_g( y_1,y_2)$  their exist an isometry $I : M \to M$ such that $ I(x_1) = I(y_1)$ and $ I(x_2) = I(y_2)$. These space  were fully characterized by Wang  \cite{Wang}. The complete list of then are (where we drop the Cayley projective plane since their no varying dimension to look cutoff phenomenon)  :

\begin{enumerate}
\item the $n$ dimensional sphere   $  \mathbb{S}^n$, n= 1,2,3, ...,
\item the real projective space   $\mathbb{P}^n(\mathbb{R})  $, n= 2,3,4,...,
\item the complex projective space   $\mathbb{P}^n(\mathbb{C})  $, n= 4,6,8,...,
\item the quaternionic projective space   $\mathbb{P}^n(\mathbb{H})  $, n= 8,12,16,...,
\end{enumerate}	
In the following, $ M^n$ will be one of the following space : $ \mathbb{S}^n$, $ \mathbb{P}^n(\mathbb{R})$, $ \mathbb{P}^n(\mathbb{C})$ or $ \mathbb{P}^n(\mathbb{H})$.
Let $\Delta $ be Laplace-Beltrami operator  in $ M^n$, and
let $\Delta_n  $  the radial part  of the Laplace-Beltrami operator $\Delta $. It acts on function defined over $[0, \pi] $ (recall that all $M^n $ have been normalized to have the same diameter $ \pi$), and it's well know e.g. Section 3 in  \cite{Jaming} that
 $$  \Delta_n = \partial^2_r + \ln(I_{M^n}'(r))' \partial_r ,$$
 where $ I_{M^n}'(r)$ have the following general expression (defined up to some multiplicative constant) : 
  \bqn{Idef} I_{M^n}'(r) = \sin \big( \frac{\gamma r}{2} \big)^\sigma \sin(\gamma r)^{\rho}, \eqn 
  where the parameter depend one $M^n$ as follow :

\begin{enumerate}
\item $  \mathbb{S}^n$ : $ \sigma =0$, $ \quad {\rho} = n-1$, $  \quad \gamma  = 1$, $ \quad I_{ \mathbb{S}^n}'(r)=  \sin^{n-1}(r) $
\item $\mathbb{P}^n(\mathbb{R})  $ : $ \sigma =0$, $ \quad {\rho} = n-1$, $  \quad \gamma  = \frac12$, $ \quad I_{ \mathbb{P}^n(\mathbb{R})}'(r)=  \sin^{n-1}(\frac{r}{2}) $
 \item $\mathbb{P}^n(\mathbb{C})  $ : $ \sigma =n-2$, $ \quad {\rho} = 1$, $  \quad \gamma  = 1$, $\quad  I_{\mathbb{P}^n(\mathbb{C}) }'(r)=  \sin(r)\sin^{n-2}(\frac{r}{2}) $
\item $\mathbb{P}^n(\mathbb{H})  $ : $ \sigma =n-4$, $ \quad {\rho} = 3$, $  \quad \gamma  = 1$, $\quad  I_{\mathbb{P}^n(\mathbb{H})}'(r)=  \sin^3(r)\sin^{n-4}(\frac{r}{2}) $

\end{enumerate}	
With respect to the above list, define :

  \bqn{parameter2}  \alpha =\frac{n-2}{2}, \quad \beta= \frac{\rho -1 }{2}  .\eqn 
  The eigenvalue of the radial part of the Laplace-Beltrami operator $\Delta_n $ are (after a change of variable $ t=\cos(\gamma \theta )$):
  
  \bqn{eigenvalue} \lt\{\begin{array}{lll}
  \lambda_k  &= -k(k+ \alpha + \beta +1), & \quad  k \in \mathbb{Z}_+    \quad \text{if}\quad M^n \neq \mathbb{P}^n(\mathbb{R}) \\ 
    \lambda_k & = -k(k+ \frac{n-1}{2}), & \quad  k \in \mathbb{Z}_+     \quad \text{if}\quad M^n = \mathbb{P}^n(\mathbb{R}) \\ 
  \end{array}  \rt.    
  \eqn

\begin{deff}\label{BM}
For any $n\in\NN\setminus\{1\}$,  $X_n\df(X_n(t))_{t\geq 0}$ 
 stands for the Brownian motion on  $M^{n}$ started at $\wi 0   \in M^{n}$ (an arbitrary point ) and time-accelerated by a factor $2$,  i.e.\ the $ \Delta$-diffusion in $M^{n}$.
So the generator of $X_n$ is the Laplacian $ \Delta$ and not the Laplacian divided by 2 as it is sometimes more usual in Probability Theory. 
 \end{deff}

As a consequence of the construction in \cite{zbMATH07470497}, $X_n$ can be intertwined with a process $D\df(D(t))_{t\geq 0}$ taking values in the closed balls of $M^{n}$ centered at $\wi 0$, starting at $\{\wi 0\}$ and absorbed in finite time $\tau_{M^n}$ in the whole set $M^{n}$. In fact the general intertwined process that starts at a domaine $D_0 $  satisfy the following equation (44) in  \cite{zbMATH07470497} :
\bqn{sflow}
\fo t\in[0,\uptau),\,\fo x\in \pa D_t, \qquad d \pa D_t(x)&=&\lt(\sqrt{2}dB_t+2\frac{\usm(\pa D_t)}{\mu(D_t)}dt-\rho_{\pa D_t}(x)dt\rt)\nu_{\pa D_t}(x)
\eqn
where $B $ is a one dimensional Brownian motion, $\nu_{\pa D_t}(x)$ is the exterior normal vector of ${\pa D_t} $ at $x$, $\rho_{\pa D_t}(x)$ is the mean curvature of ${\pa D_t} $ at $x$, $\mu  $ is the Riemannian measure of $M^n $ and $\usm$ is the $(n-1)-$dimensional Hausdorff measure.

Since all the manifold $M^n $ are harmonic, the mean curvature of sphere are constant and depend only on the ray, then if equation \eqref{sflow} start at a ball $B(\wi 0,R(0))\df D(0)$  then $ D_t \df D(t)$ remains a ball.   
Writing $B(\wi 0,R(t))\df D(t)$ for $t\in[0,\tau_{M^n}]$, equation \eqref{sflow} become an equation on the radius, and $R\df(R(t))_{t\in[0,\tau_{M^n}]}$ is solution of  the following  stochastic differential equation :

\bqn{R}
\fo t\in(0,\tau_{M^n}( R_0)),\qquad dR(t)&=&\sqrt{2} dB(t)+b_n(R(t)) dt \quad R(0)= R_0 \in ]0, \pi [\eqn
and 
\bqn{taun}
\tau_{M^n}( R_0)&=& \inf\{t\geq 0\st R(t)= \pi \}\eqn
where  the mapping $b_n$ is given by
\bqn{bn}
\fo r\in(0,\pi),\qquad
b_n(r) &\df& 2\frac{\usm( \partial B(\wi 0,r) )}{ \mu(B(\wi 0,r)) }- \rho_{\partial B(\wi 0,r)} \notag \\
&=& 2\frac{I_{M^n}'(r)}{ I_{M^n}(r) }- \frac{I_{M^n}''(r)}{I_{M^n}'} \notag \\
&=& \f{d}{dr} \ln \lt(\frac{I_{M^n}^2(r)}{I_{M^n}'(r)}\rt) \\
\notag
\eqn
since the volume of the geodesic ball $B(\wi 0,r) $ in $M^{n}$ centered at $\wi 0$  of radius $r\in ]0, \pi]$ is given by (e.g. page 8 of \cite{Jaming})
$$  \mu(B(\wi 0,r)) = c_n \int_0^r  I_{M^n}'(s) ds = c_n I_{M^n}(r) ,$$for some constant  $c_n$, the area of the geodesic sphere 
$$\usm( \partial B(\wi 0,r) ) = c_n  I_{M^n}'(r) $$ and the mean curvature $ \rho_{\partial B(\wi 0,r) } $ of any point in $\partial B(\wi 0,r) $ is given by 
$$  \rho_{\partial B(\wi 0,r)} =  \frac{I_{M^n}''(r)}{I_{M^n}'(r)}.$$
We could check  that as $r$ goes to $0_+$
\bq
b_n(r)&\sim &\f{n+1}{r}\eq

and this is sufficient to insure that 0 is an entrance boundary for $R$, so that starting from 0, it will never return to 0 at positive times. Until to now, $ \tau_{M^n}$ will stand for $\tau_{M^n}(0) $. Let $$L_n := \partial^2_r + b_n(r) \partial_r $$ be the generator of  $ R $   defined in \eqref{R}, in order to clarify the notation we don't write explicitly the dependence  in $M^n $ when their are no possible confusions.

In  \cite{arnaudon:hal-03037469}, several couplings of  $X_n$ and $D$ were constructed, so that for any time $t\geq 0$, the conditional law of $X_n(t)$ knowing the trajectory $D({[0,t]})\df(D(s))_{s\in[0,t]}$ is the normalized uniform law over $D(t)$, which will be denoted $\Lambda_n(D(t),\cdot)$ in the sequel, i.e.
\bqn{loiprod}
\fo t\geq 0,\qquad \cL(X_t\vert  D_{[0,t]})&=&\Lambda_n(D_t,\cdot),\eqn
\par
 where $ \Lambda_n$ is Markov kernel defined as follows : for all $ D$ domain or singleton  of $M^n$
\bqn{Lambdan}
  \fo A\in \cB(M^n),\qquad \Lambda_n(D,A)&\df &
\lt\{
\begin{array}{lcl}  
&\frac{\mu (A\cap D)}{\mu (D)} &\hbox{, if $\mu(D)>0$}\\
&\delta_x(A) & \hbox{, if $D=\{x\}$, with $x\in M^n$}.
\end{array}
\rt.
\eqn

   Furthermore, $D$ is progressively measurable with respect to $X_n$, in the sense that for any $t\geq 0$, $D({[0,t]})$ depends on $X_n$ only through $X_n({[0,t]})$.

In the following  corollary we explicit two intertwining relations, which were constructed in \cite{arnaudon:hal-03037469} Theorems 3.5 and 4.1, enabling to deduce $\tau_{M^n}$ from the Brownian motion $X_n$ (and independent randomness for the second construction). Note that the cutlocus of $\tilde 0$,  $cut(\wi 0) = \{ x \in M^n, s.t. \, d(\wi 0, x)= \pi \}$  e.g 3.35  in\cite{Besse}, and as usual $\mu(cut(\wi 0)) =0. $
\begin{cor}\label{cor1}
Consider the Brownian motion $X_n\df (X_n(t))_{t\geq 0}$ in $M_f^{n}$ described in Definition \ref{BM}. For $x\in M^{n} \backslash\{\wi 0, cut(\wi 0)\}$, let  $N(x)=-\nabla d(\wi 0,\cdot)(x)$, where $d $ is the Riemannian distance in $M^{n}$.
\begin{itemize}
\item[(1)] \textbf{Full coupling}.
 Let $D_1(t)$ be the ball in $M^{n}$ centered at $\wi 0$ with radius $R_1(t)$ solution started at $0$ to the It\^o equation
\bq
dR_1(t)&=&- \langle N(X_n(t)), dX_n(t))\rangle + \left[2 \frac{I_{M^n}''}{I_{M^n}'}(d(\wi 0,X_n(t)))-\frac{I_{M^n}''}{I_{M^n}'}(R_1(t))\right]\, dt
\eq
This evolution equation is considered up to  the hitting time  $\tau_n^{(1)}$ of $\pi$ by $R_1(t)$. 
\item[(2)] \textbf{Full decoupling, reflection of $D$ on $X_n$}.
 Let $D_2(t)$ be the ball in $M^{n}$ centered at $\wi 0$ with radius $R_2(t)$ solution started at $0$ to the It\^o equation
\bq
dR_2(t)&=&-\sqrt{2} dW_t +2dL_t^{R_2}[d(\wi 0,X_n(t))]- \frac{I_{M^n}''}{I_{M^n}'}((R_2(t))\, dt
\eq
\sloppy
where $(W_t)_{t\ge 0}$ is a real-valued Brownian motion independent of $(X_n(t))_{t\ge 0}$ and $(L_t^{R_2}[d(\wi 0,X_n)])_{t\in[0,\tau_n^{(2)}]}$ is the local time  at $0$ of the process $R_2-d(\wi 0,X_n)$. These considerations are valid up to  the hitting time  $\tau_n^{(2)}$ of $\pi$ by $R_2(t)$. \fussy
\end{itemize}

Then  for the above coupling we have
\begin{itemize}
\item[(1)]  $(X_t,D_t^{(1)})_ {0\leq t\leq \tau_n^{(1)}}$ and  $(X_t,D_t^{(2)})_ {0\leq t\leq \tau_n^{(2)}}$ satisfies \eqref{loiprod}.
\item[(2)]   $(\tau_n^{(1)}, (D_1(t))_{t\in [0,\tau_n^{(1)}]})$,  $(\tau_n^{(2)}, (D_2(t))_{t\in [0,\tau_n^{(2)}]})$ and  $(\tau_{M^n}, (D(t))_{t\in [0,\tau_{M^n}]})$ have the same law.
\end{itemize}
  In particular $\tau_n^{(1)}$ and $\tau_n^{(2)}$ satisfies Proposition~\ref{mean-var-harm},  Theorem ~\ref{moment}, and Theorem \ref{queue1} below.

\end{cor}
 \par\me

Due to these couplings and to general arguments from Diaconis and Fill \cite{MR1071805} for discrete case and for the continuous case see Proposition
\ref{sep} below, $\tau_{M^n}$ has the same law  as a strong stationary time $ \tau_n^{(1)}$ for $X_n$, meaning that: $\tau_n^{(1)}$ is an $\mathcal{F}^{X_n}$-stopping time, $\tau_n^{(1)}$ and $X_n(\tau_{\tau_n^{(1)}})$ are independent and $X_n(\tau_n^{(1)})$ is uniformly distributed over  $M^{n}$. As a consequence we have 
\bq
\fo t\geq 0,\qquad \fs(\cL(X_n(t)),\cU_{M^n})&\leq & \PP[\tau_n^{(1)}>t] = \PP[\tau_{M^n}>t]\eq
where  the l.h.s.\ is the separation discrepancy between the law of $X_n(t)$ and the uniform distribution $\cU_{M^n}$ over $M^{n}$. Notice that $\cU_{M^n} (B_{M^n}(\wi 0,r)) = \frac{I_{M^n}(r)} {I_{M^n}(\pi)}$ for any $r\in[0,\pi]$.\par

Note that in the above equation \eqref{loiprod} it is possible to change $ t$ by any stopping time $ \tau \le \tau_n^{(1)}$ e.g. \cite{arnaudon:hal-03037469} Theorems 3.5 and 4.1, for completeness let us recall the following proof.

\begin{pro}\label{sep} If $\tau_n^{(1)}  $ is finite almost surely, then $\tau_n^{(1)}$ is a strong stationary time for $X_n$, and 
\bq
\fo t\geq 0,\qquad \fs(\cL(X_n(t)),\cU_{M^n})&\leq & \PP[\tau_n^{(1)}>t]\eq
\end{pro}
\begin{proof}
For simplicity write $X_t := X_n(t) $. 
Let $ f : M^n \to \mathbb{R} $ be a bounded measurable function,  apply \eqref{loiprod} at $ \tau_n^{(1)}$, since $D^{(1)}_{\tau_n^{(1)}} = B(\wi 0, \pi) = M^n$ we get :
\begin{align*}
\EE[ f(X_{\tau_n^{(1)}})] &=  \EE[  \EE[ f(X_{\tau_n^{(1)}}) \vert  D^{(1)}_{[0,\tau_n^{(1)}]}  ]\\
&=   \EE[  \Lambda_n(D^{(1)}_{\tau_n^{(1)}}, f)]\\
&= \EE[ \Lambda_n (M^n, f) ] \\
&=  \frac{\int_{M^n} f \, d\mu }{\mu(M^n)}   = \cU_{M^n}(f) .\\
\end{align*}
 Hence $X_{\tau_n^{(1)}}$ is uniformly distributed over  $M^{n}$, and clearly $\tau_n^{(1)}$ and $X_{\tau_n^{(1)}}$ are independent. Also let $ f : M^n \to \mathbb{R}_+ $ be a bounded positive measurable function,   and  $  0 \le t$,  we have :
 \begin{align*}
\EE_{\wi 0}[ f(X_t)] & \ge  \EE_{(\wi 0, 0)}[ f(X_t) \un_{\tau_n^{(1)} \le t} ] \\ 
 & =  \EE_{(\wi 0, 0 )}[ \un_{\tau_n^{(1)} \le t} \EE [f(X_t) \vert \mathcal{F}_{\tau_n^{(1) }} ]]  \\ 
&=  \EE_{(\wi 0, 0)}[ \un_{\tau_n^{(1)} \le t} \EE_{X_{\tau_n^{(1) } }} [f(X_{t-\tau_n^{(1) }})] ]  \quad \text{Strong Markov property} \\ 
&= \cU_{M^n}(f) \mathbb{P}_{0} [\tau_n^{(1)} \le t ],\\ 
\end{align*}
where in the last line we use that $ X_{\tau_n^{(1) }} \sim\cU_{M^n}$   is invariant under  $X$ and $  X_{\tau_n^{(1) }}$ and $\tau_n^{(1)}$  are independent.

Let $y \in M^n$, $\epsilon >0$. Let  $f= \un_{B(y, \epsilon)}   $, and let $p_t( \wi 0 ,x) $ be the heat kernel i.e. $X_t(\wi 0) \sim p_t( \wi 0 ,x) \cU_{M^n} (dx)$,  then the  above inequality become
$$ \frac{  \int  p_t( \wi 0 ,x) \un_{B(y, \epsilon)}  (x) \cU_{M^n} (dx) }{ \cU_{M^n}  (B(y, \epsilon))} \ge \mathbb{P}_{0} [\tau_n^{(1)} \le t ], $$ 
letting $ \epsilon $ goes to $0$, we obtain for all $y \in M^n$,
$$ p_t( \wi 0 ,y)  \ge \mathbb{P}_{0} [\tau_n^{(1)} \le t ],$$
hence 
$$ 1 - p_t( \wi 0 ,y) \le  \mathbb{P}_{0} [\tau_n^{(1)} > t ].$$
and by definition of the separation discrepancy \eqref{sepdef} the result follows.

\end{proof}

\par
\begin{rem}

Note that for any $t\in [0,\tau_n^{(1)})$, the ``opposite pole'' $cut(\wi 0)$ does not belong to the support of $\Lambda(D^{(1)}(t),\cdot)$.
It follows from an extension of Remark 2.39 of Diaconis and Fill \cite{MR1071805} that $\tau_n^{(1)}$ is even a sharp  strong stationary time for $X_n$, meaning
that
\bq
\fo t\geq 0,\qquad \fs(\cL(X_n(t)),\cU_{M^n})&= & \PP[\tau_n^{(1)}>t]\eq
\par
Thus for all $M^n$ the understanding the convergence in separation of $X_n$ toward $\cU_{M^n}$ amounts to understanding the distribution of $\tau_n^{(1)}$ and due to the equality in law the distribution of $\tau_{M^n}$.
\end{rem}

\subsection{Covering time of the dual process using intertwining}

In this section, we will compute the cumulative distribution function of $\tau_{M^n}$.

\begin{pro}\label{Poisson-harm}
Given $ g \in C_b([0, \pi])$, the bounded solution $\phi_n$ of the Poisson equation
\bq
 \lt\{\begin{array}{rcl}
  L_n \phi_ n &=& -g  \\ 
  \phi_n (\pi) &=& 0 \\
\end{array}\rt.\eq
is given by:
 \bqn{phinrn}
 \fo r\in[0,\pi],\qquad
  \phi_n (r) &=&  \int_r^{\pi}  \frac{I_{M^n}'(t)}{I_{M^n}^2(t)}  \lt(\int_0^t  \frac{I_{M^n}^2(s)}{I_{M^n}'(s)} g(s) ds \rt) dt.  
\eqn
    So the Green operator $G_n$ associated to $L_n $ is given by 
    \bqn{gint} 
    \fo g\in C_b([0, \pi]),\,\fo r\in[0,\pi],\qquad
    G_n[g](r) = \int_r^{\pi}  \frac{I_{M^n}'(t)}{I_{M^n}^2(t)}  \lt(\int_0^t  \frac{I_{M^n}^2(s)}{I_{M^n}'(s)} g(s) ds \rt) dt.
    \eqn
 \end{pro}
   
   \begin{proof} Let us justify integrability of $ [0,\pi] \ni  t \mapsto\frac{I_{M^n}'(t)}{I_{M^n}^2(t)}   \int_0^t  \frac{I_{M^n}^2(s)}{I_{M^n}'(s)} g(s)  $ at $ 0$ and $ \pi$.
   
   Since up to some constant $c$ that could change from one line to the other,  $ I_{M^n}'(s) \sim_{s \ri 0_+} c s^{n-1}$, where $ n$ be the real dimension of $ M^n$,   we have $I_{M^n}(s) \sim_{s \ri 0_+} c s^{n} , $ so $\frac{I_{M^n}^2(s)}{I_{M^n}'(s)} \sim_{s \ri 0_+ } cs^{n+1}$ hence $ t \mapsto \frac{I_{M^n}'(t)}{I_{M^n}^2(t)}  \int_0^t  \frac{I_{M^n}^2(s)}{I_{M^n}'(s)}  ds $ is integrable at $0$. \\
   At $ \pi $,   $ I_{M^n}(\pi)$  is finite, and $I_{M^n}'(s) \sim_{s \ri \pi- } c(\pi-s)^{\alpha}$ for some $\alpha \in \{0; 1; 3 ; n-1 \} $ depending on $M^n$.
 
 For $\alpha \in \{ 3 ; n-1 \} $  we have  $\int_{\epsilon}^t  \frac{1}{I_{M^n}'(s)}  ds \sim_{t \ri \pi- }  c (\pi-t)^{1-\alpha}$ and so $I_{M^n}'(t)\int_{\epsilon}^t  \frac{1}{I_{M^n}'(s)}  ds \sim_{t \ri \pi- }  c (\pi-t) $. 
 
 For $\alpha =1 $ we have $I_{M^n}'(t)\int_{\epsilon}^t  \frac{1}{I_{M^n}'(s)}  ds \sim_{t \ri \pi- }  c (\pi-t)\ln(\pi-t) $, hence in any case  $ t \mapsto \frac{I_{M^n}'(t)}{I_{M^n}^2(t)}  \int_0^t  \frac{I_{M^n}^2(s)}{I_{M^n}'(s)}  ds $ is integrable at $\pi$.
   
For the function defined in \eqref{phinrn}, we clearly have, $\phi_n (\pi) =0 $, and for any $r\in[0,\pi]$,
   $$ \phi_n' (r) = - \frac{I_{M^n}'(r)}{I_{M^n}^2(r)} \int_0^r  \frac{I_{M^n}^2(s)}{I_{M^n}'(s)} g(s) ds  $$
   $$ \phi_n ''(r) = - \lt(\frac{I_{M^n}'}{I_{M^n}^2}\rt)' (r) \int_0^r  \frac{I_{M^n}^2(s)}{I_{M^n}'(s)} g(s) ds  - g(r) .$$
 It follows that  \begin{align*} 
   L_n \phi_ n (r) &= -g(r) - \lt(\frac{I_{M^n}'}{I_{M^n}^2}\rt)'(r) \int_0^r  \frac{I_{M^n}^2(s)}{I_{M^n}'(s)} g(s) ds   + \lt( \ln \frac{I_{M^n}^2}{I_{M^n}'}\rt)'(r) \phi_n' (r) \\
   &= -g(r) - \lt(\frac{I_{M^n}'}{I_{M^n}^2}\rt)' (r)\int_0^r  \frac{I_{M^n}^2(s)}{I_{M^n}'(s)} g(s) ds -  \lt(  \ln \frac{I_{M^n}'}{I_{M^n}^2} \rt)'(r) \lt(- \frac{f^{n-1}(r)}{I_{M^n}^2(r)} \int_0^r  \frac{I_{M^n}^2(s)}{I_{M^n}'(s)} g(s) ds  \rt) \\
   &=-g(r).\\
   \end{align*}
 \end{proof}   


As in the proof of Proposition \ref{mean-eq} let $u_{n,0}\df \un $, the constant function taking the value 1 on $[0,\pi]$, and consider the following sequence  $(u_{n,k})_{k\in\NN} $, defined inductively by bounded solution of

 \bqn{def-u-harm}
\fo k\in\NN,\qquad \lt\{\begin{array}{rcl}
 L_n u_{n,k} &=& -k   u_{n,k-1}\\
    u_{n,k} (\pi) &=& 0 .\\
\end{array}\rt.\eqn

 We have for all $n \ge 2$ , and $k \in \mathbb{Z}_+$,  $$ \frac{u_{n,k}}{k !} =  G_n^{\circ k}[\un]\df G_n[G_n[\cdots [G_n [\un]]\cdots]]. $$

\begin{pro}\label{mean-var-harm}
For any $k\in\ZZ_+$, $ n\ge 2 $, and all  $M^n $,
 \bq  \EE[\tau_{M^n}^{k}] &=&  k! G_n^{\circ k}[\un] (0)  \eq
 where the explicit dependence of  $G_n$ in term of $M^n $ is given by   $$G_n[g](r) = \int_r^{\pi}  \frac{I_{M^n}'(t)}{I_{M^n}^2(t)}  \lt(\int_0^t  \frac{I_{M^n}^2(s)}{I_{M^n}'(s)} g(s) ds \rt) dt.$$
 
\end{pro}
\proof The proof is similar to the proof of Proposition \ref{mean-eq}.

\wwtbp

\begin{rem}
Using Proposition  \ref{Poisson-harm}, and \ref{mean-var-harm} we have 
$$ \EE[\tau_{M^n}] = G_n[1](0)=  \Vert G_n[1] \Vert_{\infty} < \infty, $$
this is enough to apply Proposition \ref{sep}.

Moreover we have 
$$ \EE[\tau_{M^n}^k] \le k!  \Vert G_n[1] \Vert_{\infty}^k = k! \EE[\tau_{M^n}] ^k.$$
\end{rem}
Let $\mu_n (ds) :=   \frac{I_{M^n}^2(s)}{I_{M^n}'(s)}  ds $ be the invariant measure of $ L_n$. The operator $ G_n$ could be write, using Fubini,  in term of symmetric kernel $$ k_n(r,s) : = \int_{r \vee s}^{\pi } \frac{I_{M^n}'(t)}{I_{M^n}^2(t)}  \,dt,$$
where $ r \vee s = \max \{r,s \}$ and \bqn{gker} G_n[g](r) = \int_0^{\pi}  k_n(r,s) g(s) \mu_n(ds). \eqn
Contrary to the torus case, $ \un  \notin L^2([0, \pi],\mu_n), $  (for $ M^n  \neq \mathbb{P}^n(\mathbb{R}) $)  nevertheless we have the following Proposition.
\begin{pro} \label{gcirc}
For all $ k \ge \lceil \frac{n}{2} \rceil+ 2$, we have $$ s \mapsto G_n^{\circ (k  )}[\un] (s) \frac{I_{M^n}^2(s)}{I_{M^n}'(s)}  $$ is bounded in $[0, \pi] $,
 and $$G_n^{\circ k}[\un] \in  L^2([0, \pi],\mu_n).$$
\end{pro}
\begin{proof}
\begin{enumerate}[(i)]

\item
For $g \in  C_b([0, \pi])$, we want to get sufficient condition to have $ g (s) \frac{I_{M^n}^2(s)}{I_{M^n}'(s)}  $  bounded in $[0, \pi] $  and $ g \in L^2([0, \pi],\mu_n) $, i.e.  $ \int_0^{\pi}  g^2(s) \mu_n(ds) = \int_0^{\pi}  g^2(s) \frac{I_{M^n}^2(s)}{I_{M^n}'(s)} ,ds < \infty$. Since at  $0  $ their are no problem of integrability and boundedness as in the proof of Proposition \ref{mean-var-harm} we have to look  the integrability and boundedness at $ \pi$.

At $ \pi $, since    $ I_{M^n}(\pi)$  is finite, and $I_{M^n}'(s) \sim_{s \ri \pi- } c(\pi-s)^{\alpha}$ for some $\alpha \in \{0; 1; 3 ; n-1 \} $, to get $ g \in L^2([0, \pi],\mu_n) $ it is enough to have $ g  \sim_{s \ri \pi- } c(\pi-s)^{\beta}$ with $\frac{\alpha -1}{2} < \beta$, in all case it is enough to have $ \lceil \frac{n}{2} \rceil \le {\beta}$, and in order to get boundedness of $ g (s) \frac{I_{M^n}^2(s)}{I_{M^n}'(s)}  $ it is enough to have $ g  \sim_{s \ri \pi- } c(\pi-s)^{\beta}$ with $ n-1 \le  \beta$.
\item
 In what follows we will look at the worst case i.e. the sphere case $ \alpha = n-1$ the others will done in the same way.

In this case $I_{M^n}'(s) \sim_{s \ri \pi- } (\pi-s)^{n-1}$, if  $g \sim_{s \ri \pi- } c(\pi-s)^{\gamma}$  with $ \gamma < n-2$. Since 
 $$ G_n[g](r) = \int_r^{\pi}  \frac{I_{M^n}'(t)}{I_{M^n}^2(t)}  \lt(\int_0^t  \frac{I_{M^n}^2(s)}{I_{M^n}'(s)} g(s) ds \rt) dt,$$
 and $\int_0^t  \frac{I_{M^n}^2(s)}{I_{M^n}'(s)} g(s) ds \sim_{t \ri \pi- } c \int_0^t  \frac{1}{ (\pi-s)^{n-1 -\gamma}}  ds  \sim_{t \ri \pi- } \frac{c}{ (\pi-t)^{n-2 -\gamma}},$
 hence  $$G_n[g](r) \sim_{r \ri \pi- } c\int_r^{\pi}  (\pi-t)^{n-1}   \frac{1}{(\pi-s)^{n-2 -\gamma}}  dt \sim_{t \ri \pi- } c(\pi-r)^{\gamma+2}.$$
 If $ \gamma = n-2$ then $$G_n[g](r) \sim_{r \ri \pi- } -c (\pi-r)^{n} \ln (\pi-r) ,$$ and 
 $G_n^{\circ 2}[g](r)  \sim_{r \ri \pi- } c (\pi-r)^{n}  $.
 
 If $ \gamma > n-2$ then $$G_n[g](r) \sim_{r \ri \pi- } c (\pi-r)^{n}.$$
 
 Hence, using $ (ii)$ recursively since $ \un (r) \sim_{r \ri \pi- } (\pi-r)^{0}$, we get that if $2\beta \ge n-2 $ then $ G_n^{\circ (\beta +2 )}[\un](r)  \sim_{r \ri \pi- } c (\pi-r)^{n}  $. Hence if $ k \ge \lceil \frac{n}{2} \rceil + 2 $ then $  G_n^{\circ (k )}[\un](r)  \sim_{r \ri \pi- } c (\pi-r)^{n}  $ and so using  $i)$ we have  $ G_n^{\circ (k  )}[\un] \in L^2([0, \pi],\mu_n)  $ and $ G_n^{\circ (k  )}[\un] (s) \frac{I_{M^n}^2(s)}{I_{M^n}'(s)}  $ is bounded in $[0, \pi] $.

\end{enumerate}
 \end{proof}
 
 \begin{pro} \label{Gdiag}
 The operator $G_n $ is also an operator acting on $ L^2([0, \pi],\mu_n)$:
 $$ G_n :  L^2([0, \pi],\mu_n) \to  L^2([0, \pi],\mu_n) .$$
 Moreover $G_n$ is self-adjoint  and compact, and so there exists a countably infinite orthonormal basis of $L^2([0, \pi],\mu_n) $ consisting of eigenvectors of $G_n$, with corresponding eigenvalues $\nu_k \in \mathbb{R}^* $.
 \end{pro}
 \begin{proof}
 We will show that $G_n$ is an Hilbert–Schmidt operator. Recall that 
   \bq  G_n[g](r) = \int_0^{\pi}  k_n(r,s) g(s) \mu_n(ds), \eq
   where $$ k_n(r,s) : = \int_{r \vee s}^{\pi } \frac{I_{M^n}'(t)}{I_{M^n}^2(t)}  \,dt.$$
   
Hence $$ (k_n(r,s))^2 \le  \big( \int_{r}^{\pi } \frac{I_{M^n}'(t)}{I_{M^n}^2(t)}  \,dt \big) \big( \int_{s}^{\pi } \frac{I_{M^n}'(t)}{I_{M^n}^2(t)}  \,dt \big).$$ 
It follows that
\begin{align*}
\int_{0}^{\pi } \int_{0}^{\pi } (k_n(r,s))^2 \mu_n(ds)\mu_n(dr) &\le  \int_{0}^{\pi } \int_{0}^{\pi }    \big( \int_{r}^{\pi } \frac{I_{M^n}'(t)}{I_{M^n}^2(t)}  \,dt \big) \big( \int_{s}^{\pi } \frac{I_{M^n}'(t)}{I_{M^n}^2(t)}  \,dt \big)          \mu_n(ds)\mu_n(dr)  \\
&=  \Big( \int_{0}^{\pi }     \int_{r}^{\pi } \frac{I_{M^n}'(t)}{I_{M^n}^2(t)}  \,dt  \mu_n(dr)            \Big)^2 \\
&= \Big( \int_{0}^{\pi }   \frac{I_{M^n}^2(r)}{I_{M^n}'(r)}  \int_{r}^{\pi }    \frac{I_{M^n}'(t)}{I_{M^n}^2(t)}  \,dt  \,dr            \Big)^2 \\
&= \Big( \int_{0}^{\pi }    \frac{I_{M^n}'(t)}{I_{M^n}^2(t)}  \int_{0}^{t}    \frac{I_{M^n}^2(r)}{I_{M^n}'(r)} \,dr  \,dt            \Big)^2 \\
&= ( G_n(\un) (0)  )^2 , \\
\end{align*}
where we use Fubini's Theorem. It follows that $ G_n$ is an  Hilbert–Schmidt operator on $L^2([0, \pi],\mu_n)$ and so $ G_n$ is a compact operator. Since $k_n $ is symmetric $G_n$ is also self dual.
Hence by Spectral Theorem,  there exists a countably infinite orthonormal basis of $L^2([0, \pi],\mu_n) $ formed by  eigenvectors of $G_n$, with corresponding eigenvalues $\nu_k \in \mathbb{R} $ is such that $ \sum_{k=0}^{\infty} \nu_k^2 < \infty$ and so $ \nu_k \to_{k \to \infty} 0$. Note also by  \eqref{gint}, $0$ could not be an eigenvalue of $G_n$ and so $\nu_k \in \mathbb{R}^* $.

 \end{proof}

 Recall that $\Delta_n  $ be the radial part  of the Laplace-Beltrami operator $\Delta $ in $ M^n$. It acts on function over  $[0, \pi] $  and  
 $$  \Delta_n = \partial^2_r + \ln(I_{M^n}'(r))' \partial_r .$$
 If $\wi{0}  $ is an arbitrary point in $ M^n$, and $d$ the Riemannian distance in $M^n$ a $ \mathcal{C}^2$ function $f : M^n \to \mathbb{R}$ is called radial at $\wi{0} $ if  their exist $F \in \mathcal{C}^2([0, \pi]) $ such that $ f(x) = F(d(\wi{0}, x) ) ,$ and for such function we have 
 \bqn{com_rad}
  \Delta f (x) = \Delta_n(F) (d(\wi{0}, x) ) 
    \eqn
 Let  $ \wi{\mu}_n (dr) =  I_{M^n}'(r) dr $ be the invariant measure of $ \Delta_n$ in $[0, \pi] $, and consider as in \cite{zbMATH07470497} the link 
 \bq
 \Lambda_n :  \mathcal{C}^2([0, \pi]) &\to &  \mathcal{C}^2([0, \pi])    \\
                   F                  &\mapsto &  r \mapsto \frac{\int_0^r F(s)   \wi{\mu}_n (ds)}{\wi{\mu}_n [0,r]} = \frac{\int_0^r F(s)   \wi{\mu}_n (ds)}{I_{M^n}(r)}.\\             
 \eq
As a remarque it could be noted that we have the following commutating diagram  :

\[ 
 \xymatrix{
 \mathcal{C}^2(M^n)\vert_{rad} \ar[r]^{\Delta} \ar[d]_{\Lambda_0} &   \mathcal{C}^2(M^n)]) \vert_{rad} \ar[d]^{\Lambda_0 }\\
  \mathcal{C}^2([0, \pi]) \ar[r]^{\Delta_n } \ar[d]_{\Lambda_n } &   \mathcal{C}^2([0, \pi]) \ar[d]^{\Lambda_n }\\
  \mathcal{C}^2([0, \pi])  \ar[r]_{L_n} & \mathcal{C}^2([0, \pi]),   
}
\]

where $\mathcal{C}^2(M^n)\vert_{rad} $ is the space of function over $ M^n$ that are radial, $ \Lambda_0 (f) (r) = \frac{\int_{S(\wi{0},r)} f(x) d\sigma}{\sigma ( S(\wi{0},r)) }$ is an orbital integral, note that if $f$ is radial such  that $ f(x) = F(d(\wi{0}, x) ) ,$ then  $\Lambda_0 (f) (r) =F(r)$ and the upper part commutative diagram follows from \eqref{com_rad}. By direct computation or as application of \cite{zbMATH07470497} we have the  commutation of the lower part of the diagram i.e. $ L_n  \Lambda_n = \Lambda_n \Delta_n  $. \\
After a change of variables ($t= \cos(\gamma r)$) the spectral decomposition of $\Delta_n $ in $L^2([0,\pi], \wi{\mu}_n) $ is express in term of Jacobi polynomial $P^{\alpha, \beta}_k $, we refer as example  to section 2 and 3 of \cite{Jaming} . Let us receptively called  $( \{ (\varphi_k, \lambda_k)_k, \quad k\in \mathbb{Z}_+\} $  the eigenfunction and eigenvalue of $\Delta_n$, and recall that $ \{ \varphi_k, k \in \mathbb{Z}_+  \}$ form a complete orthogonal system of $L^2([0,\pi], \wi{\mu}_n) $, where
$$ \varphi_k (r) = P^{\alpha,\beta}_k ( \cos (\gamma r)  ) ,$$
for $ k \in \mathbb{Z}_+ $ (and for projective space $ \varphi_k (r) = P^{\alpha,\alpha}_{2k} ( \cos (\gamma r)  ) $), and $\lambda_k $ is given by \eqref{eigenvalue}.
 
The following proposition characterize the spectral decomposition of $G_n$ in terms of the spectral decomposition of $ \Delta_n$ except we have to remove the contribution of $\varphi_0=1 $ the eigenfunction associated to $ \lambda_0 = 0$. 

\begin{pro}\label{iso-spec}
The family  $ \{ \Lambda_n \varphi_k , \quad    k \in \mathbb{Z}^+\}$  is an orthogonal complete  system of functions in $L^2([0;\pi], \mu_n(dr)) $ where $ \mu_n(dr)  $ is the invariant measure of $ L_n$. 
 Also  $ \{ \wi {\Lambda_n \varphi_k} := \frac{\Lambda_n \varphi_k}{ \Vert \Lambda_n \varphi_k \Vert_{L^2([0;\pi], {\mu}_n)} }, k \in \mathbb{Z}^+  \} $ are an Hilbert basis of  eigenfunction of $ G_n $ associated respectively to the eigenvalue $ \frac{-1}{\lambda_k}$, where $\Vert  \Lambda_n\varphi_k  \Vert^2_{L^2([0;\pi], \mu_n)} = -\frac{1}{\lambda_k} \Vert \varphi_k \Vert^2_{L^2([0;\pi], \wi{\mu}_n)} $ for $ k \in \mathbb{Z}^+$ and $\lambda_k $ is given by \eqref{eigenvalue}. We have the following spectral decomposition, for all  $  g \in L^2([0, \pi],\mu_n) $ and for all $ m \in \mathbb{Z}_+ $
 \bqn{spectdes}
 G^{\circ m }_n(g) = \sum_{k=1}^{\infty} ( - \frac{1}{ \lambda_k}  )^m \frac{1}{\Vert \Lambda_n \varphi_k \Vert^2_{L^2([0;\pi], {\mu}_n)}}\Lambda_n \varphi_k \langle \Lambda_n \varphi_k, g \rangle_{L^2([0;\pi], {\mu}_n)}
 \eqn

\end{pro}
\begin{proof}
Let $ k \in \mathbb{Z}^+ $, recall that $ -\lambda_k > 0$. At first step we will compute $\Lambda_n \varphi_k(r) = \frac{\int_0^r \varphi_k (s)   \wi{\mu}_n (ds)}{I_{M^n}(r)}$.
\begin{align*}
\int_0^r \varphi_k (s)   \wi{\mu}_n (ds) &= \int_0^r \varphi_k (s)   I_{M^n}'(s) ds \\
&= \frac{1}{\lambda_k} \int_0^r \Delta_n \varphi_k (s)   I_{M^n}'(s) ds \\
&=        \frac{1}{\lambda_k} \int_0^r      \big( \varphi_k''(s) + \ln(I_{M^n}'(s))'  \varphi_k' (s) \big)  I_{M^n}'(s) ds\\
&=      \frac{1}{\lambda_k} \int_0^r       ( \varphi_k'(s)I_{M^n}'(s))'   ds\\
&= \frac{1}{\lambda_k}\varphi_k'(r)I_{M^n}'(r)
\end{align*}
where in the last line we have used that $ I_{M^n}'(0)=0$. It follows that 
\bqn{Lambda-phi}  \Lambda_n \varphi_k(r) = \frac{1}{\lambda_k} \frac{\varphi_k'(r)I_{M^n}'(r)}{I_{M^n}(r)} .\eqn
Now  we show that $\Lambda_n \varphi_k \in L^2([0;\pi], \mu_n ) ,$ 
 and $(\Lambda_n \varphi_k)_{k \in \mathbb{Z}^+}$ is an orthogonal system of functions in $ L^2([0;\pi], \mu_n ) $ . 
Let $k,j \in \mathbb{Z}^+ $,
\begin{align*}
\langle  \Lambda_n\varphi_k ,\Lambda_n\varphi_j \rangle_{L^2([0;\pi], \mu_n )} &= \int_0^{\pi} \Lambda_n\varphi_k (s) \Lambda_n\varphi_j (s)\frac{I_{M^n}^2(s)}{I_{M^n}'(s)}  \,ds \\
&= \frac{1}{\lambda_k \lambda_j}\int_0^{\pi} \varphi_k'(s) \varphi_j'(s)I_{M^n}'(s)  \,ds \\
&= \frac{1}{\lambda_k \lambda_j} \Big( \big[\varphi_k \varphi_j' I_{M^n}'\big]_0^{\pi} -  \int_0^{\pi} \varphi_k(s) \big(\varphi_j'(s) I_{M^n}'(s)\big)'  \,ds    \Big)\\
&= -\frac{1}{\lambda_k \lambda_j} \int_0^{\pi} \varphi_k(s) \big( \varphi_j''(s) + \ln(I_{M^n}'(s))'  \varphi_j'(s)    \big) I_{M^n}'(s) \,ds\\
&= -\frac{1}{\lambda_k \lambda_j}  \int_0^{\pi} \varphi_k(s) \Delta_n \varphi_j (s)I_{M^n}'(s) \,ds \\
&= -\frac{1}{\lambda_k} \langle  \varphi_k ,\varphi_j \rangle_{L^2([0;\pi], \wi{\mu}_n )},
\end{align*}
where we use \eqref{Lambda-phi} in the second line and  in the forth line we use that :
\begin{itemize}
\item $I_{M^n}'(0)= I_{M^n}'(\pi)=0 $ for $ M^n  \neq \mathbb{P}^n(\mathbb{R}) ,$
\item  $I_{M^n}'(0)=0$ and $ \varphi_j'(\pi) =0 $ (for $ M^n  = \mathbb{P}^n(\mathbb{R}) $, since in this case $\varphi_j(r) = P^{\alpha,\alpha}_{2j} (\cos( \frac{r}{2}) )   $ and $(P^{\alpha,\alpha}_{2j})'(0)=0$ using e.g. (2.9) in \cite{Jaming}).
\end{itemize} 

Since $ (\varphi_k)_{k \in \mathbb{Z}^+} $ is orthogonal in $L^2([0;\pi], \wi{\mu}_n ) $ we get that
 $(\Lambda_n \varphi_k)_{k \in \mathbb{Z}^+}$ is an orthogonal system of functions in $ L^2([0;\pi], \mu_n(dr)) $, also :
\bqn{Lnorm} \Vert  \Lambda_n\varphi_k  \Vert^2_{L^2([0;\pi], \mu_n )} = -\frac{1}{\lambda_k} \Vert \varphi_k \Vert^2_{L^2([0;\pi], \wi{\mu}_n )}.\eqn 

By the commutation of the above diagram we have for $ k \in \mathbb{Z}^+ $,
$$ L_n \Lambda_n\varphi_k  = \Lambda_n L_n \varphi_k = \lambda_k \Lambda_n \varphi_k $$ hence 
$$ L_n (- \frac{1}{\lambda_k} \Lambda_n\varphi_k)  = - \Lambda_n \varphi_k .$$
Since $ \Lambda_n \varphi_k (\pi) = \frac{\int_0^{\pi} \varphi_k (s)   \wi{\mu}_n (ds)}{I_{M^n}(\pi)} = 0$ (this is not true for $ k=0$), and $ \Lambda_n \varphi_k  \in  C_b([0, \pi])$ we have by Proposition \ref{Poisson-harm} that 
$$G_n (\Lambda_n \varphi_k) =  - \frac{1}{\lambda_k} \Lambda_n\varphi_k$$
and so $\Lambda_n \varphi_k $ are eigenfunction of $G_n $ associated with eigenvalue $- \frac{1}{\lambda_k} .$

 Let us  shows  that all eigenvector of $G_n$ are of the type $\Lambda_n \varphi $ for $ \varphi$ an eigenfunction of $\Delta_n $.
 Let $ \psi \in L^2([0;\pi], \wi{\mu}_n ) $ and $\nu \in \mathbb{R}^*$ such that $G_n (\psi) = \nu \psi$, then using \eqref{gint} $\psi  $  is enough regular and $L_n(\nu \psi)= - \psi $ and so 
 \bqn{Ln} L_n(\psi)= - \frac{1}{\nu} \psi. \eqn
 Let \bqn{tildephi} \wi{\psi}(r) = \frac{\big(I_{M^n}(r) \psi(r) \big)'}{I'_{M^n}(r)}= \psi(r) + \frac{I_{M^n}(r)  }{I'_{M^n}(r)}\psi'(r) ,\eqn
  since  $ \wi{\mu}_n (dr) =  I_{M^n}'(r) dr $ and $I_{M^n}(0)= 0 $ we have $$ \psi(r) = \Lambda_n \wi{\psi} (r).$$
 Now we show that $\wi{\psi} $ is an eigenvector of $\Delta_n $ with eigenvalue $ -\frac{1}{\nu} $. \\
 Recall that $L_n := \partial^2_r + \lt(\ln \Big( \frac{I_{M^n}^2(r)}{I_{M^n}'(r)} \Big)\rt)'\partial_r $,  so for all $ r\in[0, \pi]$ \eqref{Ln} become
 $$\psi''(r) + \lt(\ln \Big( \frac{I_{M^n}^2(r)}{I_{M^n}'(r)} \Big)\rt)' \psi'(r) = - \frac{1}{\nu} \psi(r).$$  We have by derivation of \eqref{tildephi}
 \begin{align*}
 \wi{\psi}'(r) &= 2 \psi'(r) + \frac{I_{M^n}(r)}{I_{M^n}'(r)}\psi''(r)-   \frac{I_{M^n}(r)I_{M^n}''(r)}{(I_{M^n}'(r))^2}   \psi'(r) \\
 &= 2 \psi'(r) + \frac{I_{M^n}(r)}{I_{M^n}'(r)} \lt( - \frac{1}{\nu} \psi(r) - 2 \frac{I'_{M^n}(r)}{I_{M^n}(r)}\psi'(r) + \frac{I_{M^n}''(r)}{I_{M^n}'(r)}  \psi'(r)\rt)-   \frac{I_{M^n}(r)I_{M^n}''(r)}{(I_{M^n}'(r))^2}   \psi'(r) \\
 &=- \frac{1}{\nu} \frac{I_{M^n}(r)}{I_{M^n}'(r)}\psi(r)\\
 \end{align*}
 and
 \begin{align*}
 \wi{\psi}''(r) &=- \frac{1}{\nu} \Big(\psi(r) + \frac{I_{M^n}(r)}{I_{M^n}'(r)}\psi'(r) - \frac{I_{M^n}(r)I_{M^n}''(r)}{(I_{M^n}'(r))^2}\psi(r) \Big).\\
 \end{align*}
 It follows that
 \begin{align*}
 \Delta_n \wi{\psi}(r) &=  \wi{\psi}''(r) + \ln(I_{M^n}'(r))'\wi{\psi}'(r)   \\
 &=- \frac{1}{\nu} \Big(\psi(r) + \frac{I_{M^n}(r)}{I_{M^n}'(r)}\psi'(r) - \frac{I_{M^n}(r)I_{M^n}''(r)}{(I_{M^n}'(r))^2}\psi(r) \Big)\\
 &  - \frac{1}{\nu}\frac{I''_{M^n}(r)}{I_{M^n}'(r)} \frac{I_{M^n}(r)}{I_{M^n}'(r)}\psi(r) \\
 &=- \frac{1}{\nu} \Big(\psi(r) + \frac{I_{M^n}(r)}{I_{M^n}'(r)}\psi'(r) \Big)\\
 &=- \frac{1}{\nu}\wi{\psi}(r)
 \end{align*}
 Hence all the eigenfunction of $G_n $ are of the type  $\Lambda_n \varphi_k $ with eigenvalue $- \frac{1}{\lambda_k}$  with $ k \in \mathbb{Z}^+$ (we have to remove $\varphi_0 $ ).
 By Proposition \ref{Gdiag} we know that there exist a countably infinite orthonormal basis of $L^2([0, \pi],\mu_n) $ consisting of eigenvectors of $ G_n$, so $ \lt( \frac{\Lambda_n \varphi_k}{ \Vert \Lambda_n \varphi_k \Vert_{L^2([0;\pi], {\mu}_n )} } \rt)_{k \in \mathbb{Z}^+} $ is an orthonormal basis of $L^2([0, \pi],\mu_n) $ of eigenvectors of $ G_n$ with associated eigenvalue $ \lt(- \frac{1}{ \lambda_k} \rt)_{k \in \mathbb{Z}^+} $.
 It follows that for all $  g \in L^2([0, \pi],\mu_n) $ and for all $ m \in \mathbb{Z}_+ $
 \bq
 G^{\circ m }_n(g) &=& \sum_{k=1}^{\infty} ( - \frac{1}{ \lambda_k}  )^m \frac{1}{\Vert \Lambda_n \varphi_k \Vert^2_{L^2([0;\pi], {\mu}_n )}}\Lambda_n \varphi_k \langle \Lambda_n \varphi_k, g \rangle_{L^2([0;\pi], {\mu}_n )} \\
 &=& \sum_{k=1}^{\infty} ( - \frac{1}{ \lambda_k}  )^{m-1} \frac{1}{\Vert  \varphi_k \Vert^2_{L^2([0;\pi], \wi{\mu}_n )}}\Lambda_n \varphi_k \langle \Lambda_n \varphi_k, g \rangle_{L^2([0;\pi], {\mu}_n )} \\
 \eq

\end{proof}   

\begin{rem}
In order to compute compute $\EE[\tau_{M^n}^{m}]  $ as in Proposition \ref{mean-var-harm} 
we could be tempted to evaluate $G^{\circ m}$ at $ \un$, unfortunately $ \un \notin {L^2([0;\pi], {\mu}_n )}$ if  $ M^n  \neq \mathbb{P}^n(\mathbb{R}) $. Nevertheless  by Proposition \ref{gcirc} we know that $ G^{\circ m'} (\un) \in {L^2([0;\pi], {\mu}_n )}$ for sufficiently large $m'$.
\end{rem} 

Since $k_n \in L^2( [0, \pi]^2,{\mu}_n \otimes {\mu}_n )  $ we also have
 $$k_n(r,s) = \sum_{k=1}^{\infty}  - \frac{1}{ \lambda_k}   \frac{1}{\Vert \Lambda_n \varphi_k \Vert^2_{L^2([0;\pi], {\mu}_n )}}\Lambda_n \varphi_k(r) \Lambda_n \varphi_k(s) \\ $$
 and for all $ m \in \mathbb{Z}_+ $
  \bqn{kn} k^{(m)}_n(r,s) &=& \sum_{k=1}^{\infty}     ( - \frac{1}{ \lambda_k}  )^m        \frac{1}{\Vert \Lambda_n \varphi_k \Vert^2_{L^2([0;\pi], {\mu}_n )}}\Lambda_n \varphi_k(r) \Lambda_n \varphi_k(s) , \notag \\ 
  &=& \sum_{k=1}^{\infty}     ( - \frac{1}{ \lambda_k}  )^{m-1}        \frac{1}{\Vert \varphi_k \Vert^2_{L^2([0;\pi], \wi{\mu}_n )}}\Lambda_n \varphi_k(r) \Lambda_n \varphi_k(s) , \\
  \notag 
  \eqn
where $ k_n^{(m)}$ are the iterated  kernel of $k_n $, defined recursively as $k_n^{(1)} (r,s) = k_n (r,s)$ and $k_n^{(m+1)} (r,s) = \int k_n^{(m)} (r,s_1)k_n (s_1,s) {\mu}_n(ds_1) $.  Note that $k^{(m)}_n(r,s)$ is the kernel associated to $ G^{\circ m}$. In the next Proposition we will use some properties of  Jacobi polynomials to deduce spectral formula for $\EE[\tau_{M^n}^{m}]  $ for large enough $ m$.

\begin{pro} \label{knseries}
For all $ m \ge  \lceil \frac{n}{2}  \rceil  + 2$ the following series convergence  uniformly  and absolutely  (in the two variables) :
\bq k^{(m)}_n(r,s)   &=& \sum_{k=1}^{\infty}     ( - \frac{1}{ \lambda_k}  )^{m-1}        \frac{1}{\Vert \varphi_k \Vert^2_{L^2([0;\pi], \wi{\mu}_n )}}\Lambda_n \varphi_k(r) \Lambda_n \varphi_k(s) , \\
\eq
\end{pro}
\begin{proof}
Recall that, we could express $(\varphi_k )_k$  an orthonormal basis of  eigenvectors of $ \Delta_n$ in term of Jacobi polynomials as
$$ \varphi_k (r) = P^{\alpha,\beta}_k ( \cos (\gamma r)  ) ,$$
for $ k \in \mathbb{Z}_+ $ (and for real projective space $ \varphi_k (r) = P^{\alpha,\alpha}_{2k} ( \cos (\gamma r)  ) $). 
Note that the definition of $\wi{\mu}_n(dr)= I_{M^n}'(r)dr$  \eqref{Idef} is up to a multiplicative constant, also changing $\wi{\mu}_n(dr)$ by a multiplicative constant change $ \mu_n$  by the same constant and the product $k_n (r,s) {\mu}_n(ds) $ does not change, hence for the computation of $ G^{\circ m }_n$ or for the convergence of the series   $k_n (r,s) $ we would change $\wi{\mu}_n(dr)$ by a multiplicative constant such that 
\bqn{normI} \Vert \varphi_k \Vert^2_{L^2([0;\pi], \wi{\mu}_n )} = \Vert  P^{\alpha,\beta}_k  \Vert^2_{L^2(E_{M^n}, (1-t )^{\alpha}(1+t)^{\beta} )dt)} ,\eqn
where  $\alpha, \beta$ are defined in \eqref{parameter2} and we use a change of variable $t=\cos(\gamma r) $, and
 \bq
E_M &=&
\lt\{\begin{array}{ll}
\lt[-1,1 \rt] & \quad \text{if} \quad M^n \neq \mathbb{P}^n(\mathbb{R}) \\
\lt[0,1 \rt]  & \quad \text{if} \quad  M^n = \mathbb{P}^n(\mathbb{R}) \\
\end{array}  \rt. \\
  \eq
Since $\Lambda_n \varphi_k(r) = \frac{\int_0^r \varphi_k (s)   \wi{\mu}_n (ds)}{I_{M^n}(r)},$ and $\alpha = \frac{n-2}{2}  \ge \beta =\frac{\rho-1}{2} >-\frac12$, 
 using (e.g. equation (2.10) in \cite{Jaming}) we get
  \bqn{Pinf} \Vert\Lambda_n \varphi_k \Vert_{\infty} \le \Vert  \varphi_k \Vert_{\infty}  =  \Vert   {P^{\alpha,\beta}_k}_{ \vert_{E_M }}\Vert_{\infty} = P^{\alpha,\beta}_k(1) = \frac{\Gamma (k+\alpha +1 ) }{k ! \Gamma (\alpha +1 )} \eqn
  and by (e.g. equation (2.7) and (2.8) in \cite{Jaming} ) we have
   \bqn{Pint}
\Vert  P^{\alpha,\beta}_k  \Vert^2_{L^2(E_{M^n}, (1-t )^{\alpha}(1+t)^{\beta} )dt)}  &=&
\lt\{\begin{array}{ll}
\frac{2^{\alpha + \beta+1} \Gamma (k+\alpha +1 )\Gamma (k+\beta +1 )}{k! (2k + \alpha +\beta +1  ) \Gamma (k+\alpha +\beta+1 )}  & \quad \text{if}\quad M^n \neq \mathbb{P}^n(\mathbb{R}) \\
\frac{4^{\alpha } \Gamma^2 (2k+\alpha +1 )}{(2k)! (4k + 2\alpha +1  ) \Gamma (2k+2\alpha + 1 )}  & \quad \text{if} \quad M^n = \mathbb{P}^n(\mathbb{R}) \\
\end{array}  \rt.  \\
\notag
  \eqn
We will now derive a condition on $m$ to assure the uniform (in the two variables) and the absolute convergence of the series of $k^{(m)}_n(r,s)$. In what follows, $ C_{\alpha,\beta}$ are some constant that could varying a line to the others.

If   $M^n \neq \mathbb{P}^n(\mathbb{R})$, we have using \eqref{normI} , \eqref{Pinf} and \eqref{Pint} :
\begin{align*}
  \vert (  \frac{ \Lambda_n \varphi_k(r) \Lambda_n \varphi_k(s)  }{\Vert \varphi_k \Vert^2_{L^2([0;\pi], \wi{\mu}_n)}} \vert &\le C_{\alpha,\beta} \frac{\Gamma (k+\alpha +1 )  (2k + \alpha +\beta +1  ) \Gamma (k+\alpha +\beta+1 ) }{ k ! \Gamma (k+\beta +1 )}  \\
  &\sim_{k \to \infty} C_{\alpha,\beta} \frac{\Gamma (k+\alpha +1 )   \Gamma (k+\alpha +\beta+1 ) }{ (k -1) !  \Gamma (k+\beta +1 )}  \\ 
  &\sim_{k \to \infty} C_{\alpha,\beta} \frac{\Gamma (k+\alpha +1 )   \Gamma (k+\alpha +\beta+1 ) }{ \Gamma(k)   \Gamma (k+\beta +1 )}  \\ 
  &\sim_{k \to \infty} C_{\alpha,\beta} \frac{ (k+\alpha +1 )^{k+\alpha + \frac12 }     (k+\alpha +\beta+1 )^{k+\alpha +\beta+\frac12} }{ (k)^{k -\frac12}   (k+\beta +1 )^{k+\beta + \frac12 }} , \\ 
   &\sim_{k \to \infty} C_{\alpha,\beta} k^{2 \alpha +1 }, \\
\end{align*}
where we have used that $ \Gamma(x) \sim_{x \to \infty}  \sqrt{2 \pi} x^{x- \frac12 } e^{-x}$.
 Recall that by \eqref{eigenvalue},  $ -\lambda_k \sim k^2$.
 Hence 
$$  \vert ( - \frac{1}{ \lambda_k}  )^{m-1}  \Lambda_n \varphi_k(r) \Lambda_n \varphi_k(s)      \frac{1}{\Vert \varphi_k \Vert^2_{L^2([0;\pi], \wi{\mu}_n)}} \vert \le C_{\alpha,\beta} \frac{ k ^{2 \alpha +1} }{k^{2(m-1)}} . $$
It follows that if $ m\ge   \lceil \frac{n}{2} \rceil  +2  $ the series of $ k_n^{(m)}(r,s)$   convergence  uniformly  and absolutely  in the two variables. With similar computation we have the same conclusion if $M^n =\mathbb{P}^n(\mathbb{R}).$
\end{proof}   
\begin{theo}\label{moment}
For all $ n\ge 2 $,  and for all $ m \ge  \lceil \frac{n}{2}  \rceil + 2$,
 \bq  \frac{\EE[\tau_{M^n}^{m}]}{ m!}  &=&  \sum_{k=1}^{\infty}     (  \frac{1}{ - \lambda_k}  )^{m} (-1)^{k+1} \mathcal{C}_k(M^n,\alpha,\beta), \eq
 where
 \bq 
 \mathcal{C}_k(M^n,\alpha,\beta) &=& \lt\{ \begin{array}{ll}
 \frac{ (2k + \alpha +\beta +1  ) \Gamma (k+\alpha +\beta+1 )}{ k ! \Gamma (\alpha +\beta+2 )} & \quad \text{if} \quad M^n \neq \mathbb{P}^n(\mathbb{R}) \\
    \frac{(4k + 2\alpha +1  ) \Gamma (2k+2\alpha + 1 )\Gamma(\alpha +1)}{ 2^{2k} k! \Gamma (k+\alpha + 1 ) \Gamma (2\alpha + 2 )}  & \quad \text{if} \quad M^n = \mathbb{P}^n(\mathbb{R}) \\
   \end{array}
 \rt.\eq
 
\end{theo}
\begin{proof}
Using Proposition \ref{mean-var-harm}   we have for all $m \ge  \lceil \frac{n}{2}  \rceil +2 $
\bq  
\frac{\EE_{r}[\tau_{M^n}^{m}]}{m!} &=&  G_n^{\circ m}[\un] (r) \\
&=&    \int_0^{\pi} k_n^{(m)} (r,s) {\mu}_n(ds)  \\
&= &   \int_0^{\pi} \sum_{k=1}^{\infty}     ( - \frac{1}{ \lambda_k}  )^{m-1}        \frac{1}{\Vert \varphi_k \Vert^2_{L^2([0;\pi], \wi{\mu}_n)}}\Lambda_n \varphi_k(r) \Lambda_n \varphi_k(s) {\mu}_n(ds) , \\ 
&= &   \sum_{k=1}^{\infty}     ( - \frac{1}{ \lambda_k}  )^{m-1}   \frac{1}{\Vert \varphi_k \Vert^2_{L^2([0;\pi], \wi{\mu}_n)}}\Lambda_n \varphi_k(r)    \int_0^{\pi}   \Lambda_n \varphi_k(s) {\mu}_n(ds) , \\
  \eq
  where  we use similar computation as in the proof of Proposition \ref{knseries} together with \eqref{Pinf}, \eqref{Pint}, \eqref{Lambda-phi} and $$   \Vert \varphi_k' \Vert_{\infty}  \le  \Vert   \big( {P^{\alpha,\beta}_k} \big)'_{ \vert_{E_M }}\Vert_{\infty}  \le c k^{\alpha +2}$$ e.g. (2.11) in \cite{Jaming}, to get
  $$  \vert ( - \frac{1}{ \lambda_k}  )^{m-1}  \Lambda_n \varphi_k(r) \Lambda_n \varphi_k(s)      \frac{1}{\Vert \varphi_k \Vert^2_{L^2([0;\pi], \wi{\mu}_n)}}  \frac{I_{M^n}^2(s)}{I_{M^n}'(s)} \vert \le C_{\alpha,\beta} \frac{ k ^{2 \alpha +3} }{k^{2(m)}}  $$
  in order to justify the dominated convergence Theorem used in the fourth line. From \eqref{Lambda-phi}
  we also get after integration by part that for all $ k \in \mathbb{Z}^+$ 
  
  \bq
   \int_0^{\pi} \Lambda_n \varphi_k(s) {\mu}_n(ds) &=&   \frac{1}{\lambda_k} \int_0^{\pi} \frac{\varphi_k'(s)I_{M^n}'(s)}{I_{M^n}(s)} {\mu}_n(ds) \\
  &= &  \frac{1}{\lambda_k}\int_0^{\pi} {\varphi_k'(s)I_{M^n}(s)} ds  \\
  &= &  \frac{1}{\lambda_k} \lt( [\varphi_k I_{M^n}]_0^{\pi} - \int_0^{\pi} \varphi_k(s)I_{M^n}'(s) ds\rt)\\
  &=& \frac{\varphi_k(\pi)I_{M^n}(\pi) }{\lambda_k} \\
   &=&  \lt\{\begin{array}{ll}
   (-1)^k\frac{P^{\beta,\alpha}_k(1) }{\lambda_k} I_{M^n}(\pi) & \quad \text{if} \quad M^n \neq \mathbb{P}^n(\mathbb{R}) \notag\\
   \frac{P^{\alpha,\alpha}_{2k}(0) }{\lambda_k} I_{M^n}(\pi) & \quad \text{if}\quad M^n = \mathbb{P}^n(\mathbb{R}) \notag\\
   \end{array}  \rt.  \\
  \eq
 where in the fourth equality we use that  $I_{M^n}(0) =0$ as well as   $\int_0^{\pi} \varphi_k(s)I_{M^n}'(s) ds = \langle \varphi_k , \varphi_0 \rangle_{L^2([0;\pi], \wi{\mu}_n )} =0$. In the last equality we use  that:
 
\begin{enumerate}

 \item if $M^n \neq \mathbb{P}^n(\mathbb{R})$, since $ \varphi_k (r) = P^{\alpha,\beta}_k ( \cos (\gamma r)  )$,  $ \gamma = 1$, and   $ P^{\alpha,\beta}_k(-t) = (-1)^k P^{\beta,\alpha}_k(t)$ c.f. (2.9) of \cite{Jaming}, then $\varphi_k (\pi) = (-1)^k P^{\beta,\alpha}_k(1)   $.
Note that  $ P^{\beta,\alpha}_k(1) = \frac{\Gamma (k+\beta +1 ) }{k ! \Gamma (\beta +1 )} $ for all $\alpha, \beta > -1 $ e.g. equation (4.21.2) in \cite{zbMATH03477793}, it is also a consequence of the following Rodrigues Formula.

\item if $M^n = \mathbb{P}^n(\mathbb{R})$, since  $ \varphi_k (r) = P^{\alpha,\alpha}_{2k} ( \cos (\gamma r)  ) $,  $ \gamma = \frac12 $, $\varphi_k (\pi)=  P^{\alpha,\alpha}_{2k}(0)$. Let us compute $P^{\alpha,\alpha}_{2k}(0)  $ for the real projective case. Using  Rodrigues Formula (e.g. page 4 \cite{Jaming} ) we obtain
\bq P^{\alpha,\alpha}_{2k}(0)  &= &\frac{1}{2^{2k} (2k)!} \frac{d^{2k}}{dt^{2k}}_{\vert_{t=0}} (1-t^2)^{2k+ \alpha} \\
&= &\frac{(-1)^k}{2^{2k} } \frac{(2k + \alpha)(2k + \alpha -1 )... (k + \alpha +1)}{k!} \\
&=& \frac{(-1)^k}{2^{2k} } \frac{\Gamma(2k + \alpha + 1)}{k! \Gamma(k + \alpha +1)} \\
 \eq

\end{enumerate} 

 From the normalization of $\wi{\mu}_n(dr) $ in \eqref{normI}, we have using \eqref{Pint}
 \bqn{I}
  I_{M^n}(\pi) &=&   \int_0^{\pi} \wi{\mu}_n(dr)  \notag \\
         &=&  \Vert  P^{\alpha,\beta}_0 \Vert^2_{L^2(E_{M^n}, (1-t )^{\alpha}(1+t)^{\beta} )dt)} \notag \\
&=& \lt\{\begin{array}{ll}
\frac{2^{\alpha + \beta+1} \Gamma (\alpha +1 )\Gamma (\beta +1 )}{ ( \alpha +\beta +1  ) \Gamma (\alpha +\beta+1 )}  & \quad \text{if}\quad M^n \neq \mathbb{P}^n(\mathbb{R}) \notag\\
\frac{4^{\alpha } \Gamma^2 (\alpha +1 )}{(  2\alpha +1  ) \Gamma (2\alpha + 1 )}  & \quad \text{if} \quad  M^n = \mathbb{P}^n(\mathbb{R}) \notag \\
\end{array}  \rt.  \\
 \eqn
 
 From the above computation we get :
 \bq
 \int_0^{\pi} \Lambda_n \varphi_k(s) {\mu}_n(ds) &=&  \lt\{\begin{array}{ll}
 \frac{ (-1)^k}{\lambda_k}  \frac{2^{\alpha + \beta+1} \Gamma (k +\beta +1 ) \Gamma (\alpha +1 )}{ k ! ( \alpha +\beta +1  ) \Gamma (\alpha +\beta+1 )}  & \quad \text{if} \quad M^n  \neq \mathbb{P}^n(\mathbb{R}) \notag\\
 \frac{(-1)^k}{\lambda_k 2^{2k}  } \frac{\Gamma(2k + \alpha + 1)}{k! \Gamma(k + \alpha +1)}\frac{4^{\alpha } \Gamma^2(\alpha +1 )}{ (  2\alpha +1  ) \Gamma (2\alpha + 1 )}  & \quad \text{if} \quad M^n = \mathbb{P}^n(\mathbb{R}) \notag \\
\end{array}  \rt.  \\
 \eq
  
 Also by \eqref{Pinf} if $M^n \neq \mathbb{P}^n(\mathbb{R})  $
 \bq \Lambda_n \varphi_k(0) &=& \lim_{r \to 0}  \Lambda_n \varphi_k(r) \\
 & =& \lim_{r \to 0} \frac{\int_0^r \varphi_k (s)   \wi{\mu}_n (ds)}{I_{M^n}(r)}\\
 & =& \varphi_k (0) = P^{\alpha,\beta}_k(1) = \frac{\Gamma (k+\alpha +1 ) }{k ! \Gamma (\alpha +1 )}\\
  \eq
  and if $M^n =\mathbb{P}^n(\mathbb{R})$  change $ k$ by $2k$ in the above formula.

  Putting the above computations together we obtain :
  \bq
    \Lambda_n \varphi_k(0)    \int_0^{\pi}   \Lambda_n \varphi_k(s) {\mu}_n(ds) &= &   \frac{ (-1)^k}{\lambda_k}  \lt\{\begin{array}{ll}
\frac{2^{\alpha + \beta+1} \Gamma (k +\beta +1 ) \Gamma (k+\alpha +1 )}{ (k !)^2 ( \alpha +\beta +1  ) \Gamma (\alpha +\beta+1 )}  & \quad \text{if} \quad M^n \neq \mathbb{P}^n(\mathbb{R}) \notag\\
\frac{4^{\alpha } \Gamma^2 (2k+ \alpha +1 )\Gamma( \alpha +1) }{ 2^{2k}(2k)! k!\Gamma (k+ \alpha +1 ) (  2\alpha +1  ) \Gamma (2\alpha + 1 )}  & \quad \text{if} \quad M^n = \mathbb{P}^n(\mathbb{R}) \notag \\
\end{array}  \rt.   \\
  \eq
  and so
    \bq
 \frac{ 1}{\Vert \varphi_k \Vert^2_{L^2([0;\pi], \wi{\mu}_n )}}   \Lambda_n \varphi_k(0)    \int_0^{\pi}   \Lambda_n \varphi_k(s) {\mu}_n(ds) &= &   \frac{ (-1)^k}{\lambda_k}  \lt\{\begin{array}{ll}
\frac{ (2k + \alpha +\beta +1  ) \Gamma (k+\alpha +\beta+1 )}{ k !( \alpha +\beta +1  ) \Gamma (\alpha +\beta+1 )}  & \quad \text{if} \quad M^n \neq \mathbb{P}^n(\mathbb{R}) \notag\\
\frac{(4k + 2\alpha +1  ) \Gamma (2k+2\alpha + 1 )\Gamma(\alpha +1)}{ 2^{2k} k! \Gamma (k+\alpha + 1 )(  2\alpha +1  ) \Gamma (2\alpha + 1 )}  & \quad \text{if} \quad  M^n = \mathbb{P}^n(\mathbb{R}) \notag \\
\end{array}  \rt.   \\
  \eq
  It's follows that if $ M^n \neq \mathbb{P}^n(\mathbb{R})$ then
  \bq  
\frac{\EE[\tau_{M^n}^{m}]}{m!}
&= &    \sum_{k=1}^{\infty}     (  \frac{1}{ - \lambda_k}  )^{m} (-1)^{k+1}  \frac{ (2k + \alpha +\beta +1  ) \Gamma (k+\alpha +\beta+1 )}{ k !( \alpha +\beta +1  ) \Gamma (\alpha +\beta+1 )}  , \\
  \eq
  
 and  if $ M^n = \mathbb{P}^n(\mathbb{R})$ then
  \bq  
\frac{\EE[\tau_{M^n}^{m}]}{m!}
&= &    \sum_{k=1}^{\infty}     (  \frac{1}{ - \lambda_k}  )^{m} (-1)^{k+1} \frac{(4k + 2\alpha +1  ) \Gamma (2k+2\alpha + 1 )\Gamma(\alpha +1)}{ 2^{2k} k! \Gamma (k+\alpha + 1 )(  2\alpha +1  ) \Gamma (2\alpha + 1 )} . 
  \eq
\end{proof}   
\newline
Let us compute the tail distribution of $ \tau_{M^n}$

\begin{theo}\label{queue1}
With the same notation of Theorem \ref{moment}, we have  for all $ x  > 0$
\bqn{queue} \mathbb{P} [ \tau_{M^n}\ge  x ] =  \sum_{k=1}^{\infty} (-1)^{k+1}  e^{\lambda_k  x} \mathcal{C}_k(M^n,\alpha,\beta).\eqn
\end{theo}
\begin{proof}
We compute the Laplace transform of $ \tau_{M^n}$, from Theorem \ref{moment} we have for $  \vert \lambda \vert < 1$:
\bqn{eqphin}
 \phi_n(\lambda) &=& \EE [\exp{(-\lambda \tau_{M^n}})] \notag \\
               &= &  \sum_{m=0}^{\infty} (-\lambda)^m   \frac{\EE_{r}[\tau_{M^n}^{m}]}{m!} \notag \\
               &=&  \sum_{m=0}^{ \lceil \frac{n}{2} \rceil +1 } (-\lambda)^m   \frac{\EE_{r}[\tau_{M^n}^{m}]}{m!} + \sum_{m=  \lceil \frac{n}{2} \rceil +2 }^{\infty} (-\lambda)^m   \frac{\EE_{r}[\tau_{M^n}^{m}]}{m!} \notag \\
               &=& P(\lambda) + \sum_{m=  \lceil \frac{n}{2} \rceil +2 }^{\infty} (-\lambda)^m    \sum_{k=1}^{\infty}     (  \frac{1}{ - \lambda_k}  )^{m} (-1)^{k+1} \mathcal{C}_k(M^n,\alpha,\beta) \notag\\
               &=&  P(\lambda) + \sum_{k=1}^{\infty}     (-1)^{k+1} \mathcal{C}_k(M^n,\alpha,\beta)   \sum_{m=  \lceil \frac{n}{2} \rceil +2 }^{\infty}   (  \frac{\lambda }{  \lambda_k}  )^{m} \notag \\
               &=&  P(\lambda) + \sum_{k=1}^{\infty}     (-1)^{k+1} \mathcal{C}_k(M^n,\alpha,\beta)   (  \frac{\lambda }{  \lambda_k}  )^ {\lceil \frac{n}{2} \rceil +2}   \big(  \frac{1}{1- \frac{\lambda }{  \lambda_k} }\big) \notag \\
               &=&  P(\lambda) + \lambda^ {\lceil \frac{n}{2} \rceil +2} \sum_{k=1}^{\infty}     (-1)^{k+1} \mathcal{C}_k(M^n,\alpha,\beta)   (  \frac{1 }{  \lambda_k}  )^ {\lceil \frac{n}{2} \rceil +2}   \big(  \frac{1}{1- \frac{\lambda }{  \lambda_k} }\big) \notag \\
               &=& P(\lambda) + \lambda^ {\lceil \frac{n}{2} \rceil +2} \sum_{k=1}^{\infty}     (-1)^{k+1} \mathcal{C}_k(M^n,\alpha,\beta)   (  \frac{1 }{  \lambda_k}  )^ {\lceil \frac{n}{2} \rceil +2}   \mathcal{L} \Big( x \to - \lambda_k e^{\lambda_k x}\Big)(\lambda) \notag \\
               &=& P(\lambda) + \lambda^ {\lceil \frac{n}{2} \rceil +2}   \mathcal{L} \Big(\Psi_n \Big)(\lambda) \notag \\
\eqn
where $P$ is a polynomial with degree $ \lceil \frac{n}{2} \rceil+1 $  such that $ P(0)=1$, $ \mathcal{L}$ is the Laplace transform and
$$ \Psi_n(x) =  \sum_{k=1}^{\infty}     (-1)^{k} \mathcal{C}_k(M^n,\alpha,\beta)   (  \frac{1 }{  \lambda_k}  )^ {\lceil \frac{n}{2} \rceil+1 }     e^{\lambda_k x}   . $$
Using Lemma \ref{lemme-Lap} we get for $ x >0$:
$$ \mathbb{P} [ \tau_{M^n}\ge x ] = \sum_{k=1}^{\infty} (-1)^{k+1}  e^{\lambda_k  x} \mathcal{C}_k(M^n,\alpha,\beta).$$

\end{proof}   

\begin{lem}\label{lemme-Lap} If $ \mathcal{L} (\mu) (\lambda) = P(\lambda) + \lambda^m \mathcal{L} (g)(\lambda) $, where $\mu$ is a probability distribution and $ P$ is a polynomial of degree at most $ m -1$, such that $ P(0)=1$ and $ m \in \mathbb{N}^*$ then for $x >0$ we have
$$ \int_x^{\infty} \mu(dy) = -g^{(m-1)} (x). $$
\end{lem}
\begin{proof}
Let  $P(\lambda) = \sum_{n=0}^{m-1} a_n \lambda^n $,  $P(0)=a_0=1$ and
 \bq 
 \mathcal{L} (g)(\lambda) &=&  \frac{ \mathcal{L} (\mu) (\lambda)}{\lambda^m} - \sum_{n=0}^{m-1} \frac{a_n}{ \lambda^{m-n}} \\
 &=&  \mathcal{L} (  \underbrace{\int_0^x dx_1\int_ 0^{x_1}....\int_ 0^.}_{m}  \mu(dy) ) (\lambda)- \sum_{n=0}^{m-1} a_n \mathcal{L} \Big( \frac{x^{m-n-1}H(x)}{(m-n-1)!} \Big)(\lambda) \\
 &=& \mathcal{L}  \Big( \underbrace{\int_0^x dx_1\int_ 0^{x_1}...\int_0^.}_{m}  \mu(dy)     - \sum_{n=0}^{m-1} a_n \frac{x^{m-n-1}H(x)}{(m-n-1)!} \Big)(\lambda) ,\\
 \eq
 where $ H(x) = \un_{[0, \infty[}.$
 Hence $$ g(x) =  \underbrace{\int_0^x dx_1\int_ 0^{x_1}...\int_0^.}_{m}  \mu(dy)     - \sum_{n=0}^{m-1} a_n \frac{x^{m-n-1}H(x)}{(m-n-1)!} ,$$
 after $ m -1$  derivations we get 
 $$ g^{(m-1)}(x) =  \int_0^x \mu(dy) - a_0 = - \int_x^{\infty } \mu(dy)$$
\end{proof}

\begin{rem}
Without using representation theory, it appears that $ {C}_k(\mathbb{S}^n,\alpha,\beta)$ is the dimension of the space of spherical harmonics of degree $k$, hence one can used a result from Cheeger and Yau  Lemma 2.3 in \cite{Cheeger-Yau} see also \cite{Alonso},  which states that the heat kernel is radially decreasing one the sphere. This result is not so easy to proof but is physically reasonable. Hence one can get another analytical proof of Theorem \ref{queue1} using the first definition of the separation together with the series expansion of the heat kernel.
\end{rem}

In the proof of Theorem \ref{queue1}, it appear that first $ \lceil \frac{n}{2} \rceil +1 $ moments of $ \tau_{M^n}$ is contain in $P(\lambda)$ and Lemma \ref{lemme-Lap} says in some sens that these first $ \lceil \frac{n}{2} \rceil +1 $ moments are not necessary  to characterize the law of $ \tau_{M^n}$. In fact this first moments are determined by the moments after $ \lceil \frac{n}{2} \rceil + 2 $. In the next subsection we will compute the mean of $ \tau_{M^n}$ and it's variance as example, technics could be use to compute the others moment. This mean and variance characterize the cutoff  time of the respective family of Brownian motions.  In \cite{Mag-Cou} we compute $ \EE [\tau_{M^n}] $ using different methods according to the geometries of $M^n $, as example Proposition 22, 24, 26 and 30 in \cite{Mag-Cou}. It was intriguing 
to note that this mean have the same form, in what follows we compute  these means in an unified way, and remark that this coincidence is due to the uniqueness of some different Abel type summations see Lemma \ref{Cart}.  
\subsection{Computation of the cutoff time $\EE[\tau_{M^n}] $ and the variance of $\tau_{M^n}$}

Let us compute $\EE[\tau_{M^n}] $, and the variance of $\tau_{M^n} $ when $ M^n \neq  \mathbb{P}^n(\mathbb{R})$.
On one hand we have using Theorem \ref{queue1}
\begin{align}
\EE[\tau_{M^n}]  &= \int_0^{\infty} \mathbb{P} [ \tau_{M^n}\ge x ]dx \notag \\
&= \lim_{\epsilon  \searrow 0} \int_{\epsilon}^{\infty} \mathbb{P} [ \tau_{M^n}\ge x ]dx  \notag\\
&= \lim_{\epsilon  \searrow 0} \int_{\epsilon}^{\infty} \sum_{k=1}^{\infty} (-1)^{k+1}  e^{\lambda_k  x} \mathcal{C}_k(M^n,\alpha,\beta) dx. \label{C1}\\
 \notag
\end{align}
Let $ \delta = \alpha + \beta +1 \in \mathbb{Z}_+  $, and let us write $$\lambda_k  = -k(k+ \delta),  \quad  k \in \mathbb{Z}_+    \quad \text{if}\quad M^n \neq \mathbb{P}^n(\mathbb{R}) $$
and by Theorem \ref{moment}
$$ \mathcal{C}_k(M^n,\alpha,\beta) =  \frac{ (2k + \delta) \Gamma (k+\delta )}{ \Gamma( k +1) \Gamma (\delta + 1 )}  \quad \text{if} \quad M^n \neq \mathbb{P}^n(\mathbb{R}), $$ since $\mathcal{C}_k(M^n,\alpha,\beta) $ is polynomial in $k$ we have in \eqref{C1},
\begin{align}
\EE[\tau_{M^n}]  &= \lim_{\epsilon  \searrow 0}  \sum_{k=1}^{\infty} (-1)^{k+1} \frac{ (2k + \delta) \Gamma (k+\delta )}{ \Gamma( k +1) \Gamma (\delta + 1 )} \int_{\epsilon}^{\infty} e^{ -k(k+ \delta) x} dx \notag \\
&= \lim_{\epsilon  \searrow 0}  \sum_{k=1}^{\infty} (-1)^{k-1} \frac{(2k + \delta)}{k(k+ \delta)}  \frac{  \Gamma (k+\delta )}{ \Gamma( k +1) \Gamma (\delta + 1 )} e^{ -k(k+ \delta) \epsilon} \notag \\
 &= \lim_{\epsilon  \searrow 0}  \sum_{k=1}^{\infty} a_k e^{ -k(k+ \delta) \epsilon} \label{C2},\\
 \notag
\end{align}
where $$ a_k = (-1)^{k-1} \frac{(2k + \delta)}{k(k+ \delta)}  \frac{  \Gamma (k+\delta )}{ \Gamma( k +1) \Gamma (\delta + 1 )} . $$ 
Note that we could no pass the limit inside the sum, nevertheless we know that the limit when $ \epsilon \searrow 0$ exists.  In the terminology of \cite{Hardy} page 72, $\sum a_k$ is summable  $(A, \lambda_k)$ with sum $ \EE[\tau_{M^n}].$

Inspirited by Cartwright Theorem's, specially Theorem II page 82 in \cite{Cartwright} we will look another series. Note that Theorem II \cite{Cartwright} is not directly applicable in our situation.

For $ \vert X \vert < 1 $, let 
 \begin{align}
 F(X) &=  \sum_{k=1}^{\infty} (-X)^{k-1} \frac{(2k + \delta)}{k(k+ \delta)}  \frac{  \Gamma (k+\delta )}{ \Gamma( k +1) \Gamma (\delta + 1 )}  \notag \\
 &=  \sum_{k=1}^{\infty} (-X)^{k-1} \big( \frac{1}{k} + \frac{1}{k+ \delta}  \big)\frac{  \Gamma (k+\delta )}{ \Gamma( k +1) \Gamma (\delta + 1 )}  \notag \\
 &=  \sum_{k=1}^{\infty} (-X)^{k-1} \big(\int_{0}^1 (1 + t^{\delta} ) t^{k-1} dt \big)\frac{  \Gamma (k+\delta )}{ \Gamma( k +1) \Gamma (\delta + 1 )}  \notag \\
  &=  \frac{1}{\delta}  \sum_{k=1}^{\infty} \int_{0}^1 (1 + t^{\delta} ) (-Xt)^{k-1}  \frac{  \Gamma (k+\delta )}{ \Gamma( k +1) \Gamma (\delta )}dt  \notag \\
   &=  \frac{1}{\delta}   \int_{0}^1 \big( \sum_{k=1}^{\infty} (-Xt)^{k-1}  \frac{  \Gamma (k+\delta )}{ \Gamma( k +1) \Gamma (\delta )} \big) (1 + t^{\delta} ) dt \notag  \\
   &= \frac{1}{\delta}   \int_{0}^1 \frac{1- (1 + tX)^{-\delta}  }{tX}  (1 + t^{\delta} ) dt \label{C3},\\ 
    \notag
 \end{align}
where we use dominated convergence Theorem and  $(1+x)^{-\delta} = \frac{1}{\Gamma(\delta)} \sum_{k=0}^{\infty} (-x)^{k}\frac{  \Gamma (k+\delta )}{ \Gamma( k +1) }$ for $ \vert x \vert < 1$ and $\delta >0 $, in the two last equalities. 
Hence, since $\delta \ge 1 $ 
\begin{align}
F(X) &= \frac{1}{\delta}  \Big( \int_{0}^1 \frac{1- (1 + tX)^{-\delta}  }{tX}   dt +   \int_{0}^1 t^{\delta}\frac{1- (1 + tX)^{-\delta}  }{tX}   dt \Big) \notag \\
&= \frac{1}{\delta X }  \Big( \int_{0}^X \frac{1- (1 + Y)^{-\delta}  }{Y}   dY +  \frac{1}{X^{\delta}} \int_{0}^X \big( Y^{\delta - 1 }- \frac{(Y/(1+Y))^{\delta}  }{Y}  \big)  dY \Big) \notag \\
&= \frac{1}{\delta X}  \Big( \int_{0}^X \frac{1- (1 + Y)^{-\delta}  )}{Y}   dY +  \frac{1}{{\delta}}
 - \frac{1}{X^{\delta}} \int_{0}^X  \big( \frac{Y}{1+Y} \big )^{\delta}   \frac{1}{Y}  \  dY \
\Big) \notag \\
 &\longrightarrow_{X \to 1^{-}}  \frac{1}{\delta}  \Big( \underbrace{ \int_{0}^1 \frac{1- (1 + Y)^{-\delta}  )}{Y}   dY +  \frac{1}{{\delta}}}_{A}
 -  \underbrace{\int_{0}^1  \big( \frac{Y}{1+Y} \big )^{\delta}   \frac{1}{Y}  \  dY}_{B}
\Big), \label{C4}\\
\notag
\end{align}
where we use Beppo Levi Theorem in the last line.
Let us compute $A$, recall that $\delta \in \mathbb{Z}_+ $,
\begin{align}
A &=  \int_{0}^1 \frac{1- (1 + Y)^{-\delta}  }{(1-(1+Y)^{-1}  ) (1+Y)}   dY  + \frac{1}{{\delta}}\notag  \\
&=  \int_{0}^1  \sum_{k=1}^{\delta } (1+Y)^{-k}    dY  + \frac{1}{{\delta}} \notag  \\
&= \log( 2) + \sum_{k=1}^{\delta -1} \frac{1}{k} ( 1 - \frac{1}{2^k}) + \frac{1}{{\delta}} \notag \\ 
&= \log( 2) + \sum_{k=1}^{\delta } \frac{1}{k} -  \sum_{k=1}^{\delta -1} \frac{1}{k}\frac{1}{2^k}. \label{C5}\\
\notag
\end{align}
Let us compute $ B$, after a change of variable $ Z = \frac{Y}{1+Y}$,
\begin{align}
B &= \int_0^{1/2} \frac{Z^{\delta -1}}{1-Z} dZ = \int_0^{1/2} \sum_{k=0}^{\infty} Z^{\delta -1 + k } dZ = \sum_{k=\delta}^{\infty}  \frac{1}{k} \frac{1}{2^k}.\notag \\
  \notag \\
\end{align}
Putting \eqref{C4}, \eqref{C5} in \eqref{C3}, we obtain,

$$ \lim_{X \to 1^{-}} F(X)  = \frac{1}{\delta} \sum_{k=1}^{\delta } \frac{1}{k}   . $$

Note that $$ \sum_{k=1}^{\infty} a_k e^{-k\epsilon} = e^{-\epsilon }F(e^{-\epsilon }) $$ and so
\begin{align} \label{C6}
\lim_{\epsilon \searrow 0} \sum_{k=1}^{\infty} a_k e^{-k\epsilon} &= \frac{1}{\delta} \sum_{k=1}^{\delta } \frac{1}{k} \\ \notag
\end{align}
hence  $\sum a_k$ is summable  $(A, \lambda'_k)$ with sum $ \frac{1}{\delta} \sum_{k=1}^{\delta } \frac{1}{k} $, where $ \lambda'_k = k$.

In the following we will show a Cartwright type Theorem's, that we need to identify $\EE[\tau_{M^n}]. $ We will give an alternative proof that contain our situation.

\begin{lem}\label{Cart}
If $\sum a_k$ is summable  $(A, \lambda_k)$ with sum $s1$, and if $\sum a_k$ is summable  $(A, \lambda'_k)$ with sum $s2$ then
 $$ s1 =s2.$$
\end{lem}
\begin{proof}
Let us recall the following   L\'evy's formula :
$$ \int_0^{\infty} \frac{a}{\sqrt{2\pi t^3}} e^{-\frac{a^2}{2t}} e^{-\lambda t} dt = e^{-a\sqrt{2 \lambda}}.$$
Using the above  formula we get :
\begin{align}
e^{-k \epsilon} &= \frac{ e^{\frac{\delta}{2} \epsilon}}{2 \sqrt{\pi}} \int_0^{\infty} \frac{\epsilon}{\sqrt{ t^3}} e^{-\frac{\epsilon^2}{4t}} e^{- (k+ \frac{\delta}{2} )^2 t} dt \notag \\
&= e^{\frac{\delta}{2} \epsilon} \int_0^{\infty} K_{\epsilon}(t) e^{- (k+ \frac{\delta}{2} )^2 t} dt \label{C7}\\
\notag
\end{align}
where $K_{\epsilon}(t) = \frac{\epsilon}{2 \sqrt{\pi t^3}} e^{-\frac{\epsilon^2}{4t}} $ is positive,  $  \int_0^{\infty} K_{\epsilon}(t) dt = 1$ and as the density of hitting time of standard Brownian motion at level   $ \epsilon/ \sqrt{2}$ and it concentrates at $ 0$ as $\epsilon$ goes to $0$.
Let $$ G(t) = \sum_{k=1}^{\infty} a_k e^{- (k+ \frac{\delta}{2} )^2 t} = e^{- \frac{\delta^2}{4} t} \sum_{k=1}^{\infty} a_k e^{- k(k + \delta ) t} \underset{ t \searrow 0}{\longrightarrow} s1 .$$
Also since $ G$ is clearly continuous in $ ]0, \infty[ $ and $G(t) \underset{ t \to \infty }{\longrightarrow} 0$ (by uniform convergence), $G$ is extendable by a bounded and continuous function on $ [0, \infty[ $ that we also note  $G$. 

Hence using \eqref{C7},
\begin{align}
    \sum_{k=1}^{\infty} a_k e^{-k\epsilon}  
 &=  e^{\frac{\delta}{2} \epsilon} \sum_{k=1}^{\infty} a_k \int_0^{\infty} K_{\epsilon}(t) e^{- (k+ \frac{\delta}{2} )^2 t} dt \notag \\
 &=  e^{\frac{\delta}{2} \epsilon} \int_0^{\infty} K_{\epsilon}(t)\sum_{k=1}^{\infty} a_k  e^{- (k+ \frac{\delta}{2} )^2 t} dt \notag \\
 &=  e^{\frac{\delta}{2} \epsilon} \int_0^{\infty} K_{\epsilon}(t)G(t) dt  \underset{\epsilon \searrow 0}{\longrightarrow}  G(0) =s1 ,\notag \\
 \notag
\end{align}
where in the second equality we separate even and odd $k$ and we use the positivity of $K_{\epsilon}$ as well as Beppo Levi Theorem for the commutation of sum and integral, in the last equality we use the concentration of $K_{\epsilon} \underset{\epsilon \searrow 0}{\longrightarrow} \delta_0$. The result follows since  by hypothesis
$$ \sum_{k=1}^{\infty} a_k e^{-k\epsilon} \underset{\epsilon \searrow 0}{\longrightarrow}  s2  . $$

\end{proof}

\begin{pro} \label{mean_M}
If $ M^n \neq  \mathbb{P}^n(\mathbb{R})$ then  $\EE[\tau_{M^n}] = \frac{1}{\delta} \sum_{k=1}^{\delta } \frac{1}{k}$, where 
$$ \delta =  \alpha + \beta +1 .$$ 
\end{pro}
\begin{proof}
Use \eqref{C2}, \eqref{C6} together with Lemma \ref{Cart}.
\end{proof}

\begin{cor}
We have the following equivalent :
$$\EE[\tau_{\mathbb{S}^n}] \sim_{n \to \infty} \frac{\ln(n)}{n}, $$
$$\EE[\tau_{\mathbb{P}^n(\mathbb{C})}] \sim_{n \to \infty} \frac{2\ln(n)}{n}, $$
and
$$\EE[\tau_{\mathbb{P}^n(\mathbb{H})}] \sim_{n \to \infty} \frac{2\ln(n)}{n}. $$
\end{cor}
\begin{proof}
If $M^n = \mathbb{S}^n $ then $ \delta = n-1$, and the result follows. More precise equivalent could be deduce by   Euler–Mascheroni. The other case follows in the same way.
\end{proof}
\begin{rem}
The preceding Corollary give the leading term $ t^{M^n}_n(c)$ in the next section.
\end{rem}

In a Similar way we could compute the variance of $ \tau_{M^n}$.
We have using  Theorem \ref{queue1} that :
\begin{align}
\EE[\tau_{M^n}^2] &=  2 \int_0^{\infty} x \mathbb{P} [ \tau_{M^n}\ge x ]dx \notag \\
&= 2 \lim_{\epsilon  \searrow 0} \int_{\epsilon}^{\infty} \sum_{k=1}^{\infty} (-1)^{k+1}  xe^{\lambda_k  x} \mathcal{C}_k(M^n,\alpha,\beta) dx. \notag \\
&= 2 \lim_{\epsilon  \searrow 0} \sum_{k=1}^{\infty} (-1)^{k-1}  \Big(  \frac{e^{\lambda_k \epsilon} \epsilon}{ - \lambda_k }  +  \frac{ e^{\lambda_k \epsilon} }{  \lambda_k^2 } \Big) \mathcal{C}_k(M^n,\alpha,\beta)  \notag \\
&= 2 \lim_{\epsilon  \searrow 0} \sum_{k=1}^{\infty} (-1)^{k-1}  \frac{(2k + \delta)}{(k(k+ \delta))^2}  \frac{  \Gamma (k+\delta )}{ \Gamma( k +1) \Gamma (\delta + 1 )}  e^{ -k(k+ \delta) \epsilon}    \notag \\
&= 2 \lim_{\epsilon  \searrow 0} \sum_{k=1}^{\infty} (- e^{ - \epsilon} )^{k-1}  \frac{(2k + \delta)}{(k(k+ \delta))^2}  \frac{  \Gamma (k+\delta )}{ \Gamma( k +1) \Gamma (\delta + 1 )}     \notag \\
&=2 \lim_{X  \nearrow 1} \sum_{k=1}^{\infty} (- X)^{k-1}  \frac{(2k + \delta)}{(k(k+ \delta))^2}  \frac{  \Gamma (k+\delta )}{ \Gamma( k +1) \Gamma (\delta + 1 )}   \notag   \\
&=\frac{2}{\delta} \lim_{X  \nearrow 1} \sum_{k=1}^{\infty} (- X)^{k-1} \Big( \frac{1}{k^2} - \frac{1}{(k+ \delta)^2}   \Big) \frac{  \Gamma (k+\delta )}{ \Gamma( k +1) \Gamma (\delta + 1 )}   \notag   \\
&=\frac{2}{\delta} \lim_{X  \nearrow 1} \sum_{k=1}^{\infty} (- X)^{k-1} \Big( \frac{1}{k} - \frac{1}{k+ \delta}   \Big)  \Big( \frac{1}{k} + \frac{1}{k+ \delta}   \Big)\frac{  \Gamma (k+\delta )}{ \Gamma( k +1) \Gamma (\delta + 1 )}   \notag   \\
&=\frac{2}{\delta} \lim_{X  \nearrow 1} \sum_{k=1}^{\infty} (- X)^{k-1} \Big( \int_0^1 t_1^{k-1} (1-t_1^{\delta}) dt_1 \Big)  \Big( \int_0^1 t_2^{k-1} (1+t_2^{\delta}) dt_2 \Big)\frac{  \Gamma (k+\delta )}{ \Gamma( k +1) \Gamma (\delta + 1 )}   \notag   \\
&=\frac{2}{\delta} \lim_{X  \nearrow 1} \sum_{k=1}^{\infty}  \int_0^1 \int_0^1  (- t_1t_2X)^{k-1}  (1-t_1^{\delta})   (1+t_2^{\delta}) \frac{  \Gamma (k+\delta )}{ \Gamma( k +1) \Gamma (\delta + 1 )}  dt_1 dt_2 \notag   \\
&=\frac{2}{\delta} \lim_{X  \nearrow 1}  \int_0^1 \int_0^1 \sum_{k=1}^{\infty}  (- t_1t_2X)^{k-1}  \frac{  \Gamma (k+\delta )}{ \Gamma( k +1) \Gamma (\delta + 1 )}  (1-t_1^{\delta})   (1+t_2^{\delta}) dt_1 dt_2 \notag   \\
\notag 
\end{align}
where in the fourth equality we use \eqref{C6}, in fifth equality we use Lemma \ref{Cart} (since in what follows we guarantee the existence of the limit of the second series and we compute the limit), and in the last equality we use the uniformity of the convergence for $\vert X \vert <1 $.
From  $(1+x)^{-\delta} = \frac{1}{\Gamma(\delta)} \sum_{k=0}^{\infty} (-x)^{k}\frac{  \Gamma (k+\delta )}{ \Gamma( k +1) }$ for $ \vert x \vert < 1$ and $\delta >0 $, it follows that 
\begin{align}
\EE[\tau_{M^n}^2] &= \frac{2}{\delta^2} \lim_{X  \nearrow 1}  \int_0^1 \int_0^1 (1-t_1^{\delta})  (1+t_2^{\delta})
\frac{1- \frac{1}{(1+ t_1t_2X)^{\delta}}}{t_1t_2 X}    dt_1 dt_2 \notag   \\
&= \frac{2}{\delta^2}  \int_0^1 \int_0^1 (1-t_1^{\delta})  (1+t_2^{\delta})
\frac{1- \frac{1}{(1+ t_1t_2)^{\delta}}}{t_1t_2 }    dt_1 dt_2 \notag   \\
&= \frac{2}{\delta^2}  \int_0^1 \int_0^1 (1-(t_1t_2)^{\delta})  
\frac{1- \frac{1}{(1+ t_1t_2)^{\delta}}}{t_1t_2}    dt_1 dt_2 \notag   \\
&= \frac{2}{\delta^2}  \int_0^1  \frac{1}{t_1}\int_0^{t_1} (1-x^{\delta})  
\frac{1- \frac{1}{(1+ x)^{\delta}}}{x}    dx dt_1 \label{C8}  \\
\notag 
\end{align}
where we use dominated convergence  Theorem in the second equality,  the symmetry in the third equality for the cancellation of the other term and change of variable in the last equality. By direct computation we have :
\begin{align}
\int_0^{t_1} (1-x^{\delta})  \frac{1- \frac{1}{(1+ x)^{\delta}}}{x}    dx  =  \int_0^{t_1}  \frac{1- \frac{1}{(1+ x)^{\delta}}}{x}    dx  - \frac{t_1^{\delta}}{\delta}  + \int_0^{t_1}   \Big( \frac{x}{1+ x} \Big)^{\delta} \frac{1 }{x}  dx,  \notag  \\
\notag 
\end{align}

\begin{align}
 \int_0^{t_1} \frac{1- \frac{1}{(1+ x)^{\delta}}}{x}    dx &= \sum_{k=1}^{\delta}  \int_0^{t_1} \Big( \frac{1}{1+ x} \Big)^{k} dx \notag \\
 &=\ln( 1+t_1) + \sum_{k=1}^{\delta-1} \frac{1}{k} \Big(  1 - \big( \frac{1}{1+ t_1} \big)^{k} \Big) \label{C9}\\ 
 &=  \ln( 1+t_1) - \sum_{k=1}^{\delta-1} \frac{1}{k}  \big( \frac{1}{1+ t_1} \big)^{k}  + H_{\delta -1}, \notag \\
 \notag
\end{align}
where $  H_{\delta -1}= \sum_{k=1}^{\delta-1} \frac{1}{k} $
and after a change of variable
\begin{align}
 \int_0^{t_1}   \Big( \frac{x}{1+ x} \Big)^{\delta} \frac{1}{x} dx &=  \int_{0}^{\frac{t_1}{1+ t_1}} \frac{Z^{\delta - 1}}{1-Z} dZ = \sum_{k=0}^{\infty}   \int_{0}^{\frac{t_1}{1+ t_1}} Z^{\delta - 1 +k } dZ \notag  \\
 &= \sum_{k= \delta }^{\infty}   \frac1k ({\frac{t_1}{1+ t_1}} )^{k }    \label{C10}. \\
 \notag
\end{align}
Since $$ \int_0^{1} \frac{\ln( 1+t_1)}{t_1} dt_1 = \sum_{k=0}^{\infty} (-1)^{k}\frac{1}{(k+1)^2} = \frac{\pi^2}{12}, $$
together with \eqref{C8},\eqref{C9}, and \eqref{C10}, it follows that
\begin{align}
\EE[\tau_{M^n}^2] &= \frac{2}{\delta^2} \Big(\frac{\pi^2}{12} + \sum_{k=1}^{\delta-1} \frac{1}{k} \Big( \ln(2) + \sum_{m=1}^{k-1} \frac{1}{m} ( 1 - \frac{1}{2^m}   ) \Big) - \frac{1}{\delta^2} +  \sum_{k= \delta }^{\infty}   \frac1k \sum_{m= k }^{\infty} \frac1m {\frac{1}{2^m}}   \Big) \notag \\
&= \frac{2}{\delta^2} \Big(\frac{\pi^2}{12} + \sum_{k=1}^{\delta-1} \frac{1}{k} \sum_{m=1}^{k-1} \frac{1}{m} +    \sum_{k=1}^{\delta-1} \frac{1}{k}       \sum_{m=k}^{ \infty} \frac{1}{m} \frac{1}{2^m}  - \frac{1}{\delta^2} +  \sum_{k= \delta }^{\infty}   \frac1k \sum_{m= k }^{\infty} \frac1m {\frac{1}{2^m}}   \Big) \notag \\
&=  \frac{2}{\delta^2} \Big(\frac{\pi^2}{12} + \sum_{k=1}^{\delta-1} \frac{1}{k} \sum_{m=1}^{k-1} \frac{1}{m}     - \frac{1}{\delta^2} +  \sum_{k= 1 }^{\infty}   \frac1k \sum_{m= k }^{\infty} \frac1m {\frac{1}{2^m}}   \Big) \notag \\
&=  \frac{2}{\delta^2} \Big(\frac{\pi^2}{12} + \sum_{k=1}^{\delta-1} \frac{1}{k} \sum_{m=1}^{k-1} \frac{1}{m}     - \frac{1}{\delta^2} +  \sum_{m= 1 }^{\infty} \frac1m {\frac{1}{2^m}} H_m   \Big) \notag \\
&= \frac{2}{\delta^2} \Big(\frac{\pi^2}{6} + \sum_{k=1}^{\delta-1} \frac{1}{k} \sum_{m=1}^{k-1} \frac{1}{m}     - \frac{1}{\delta^2}  \Big) \label{C11} \\
\notag
\end{align}
where in the second equality we use $\ln(2)= -\ln(1/2) = \sum_{m=1}^{\infty} \frac1m \frac{1}{2^m} $, and in the fourth  one we use Fubini Theorem and in the last one
 
\begin{align}
\sum_{m= 1 }^{\infty} \frac1m {\frac{1}{2^m}} H_m &= \sum_{m= 1 }^{\infty} \frac1m {\frac{1}{2^m}} \int_0^1 \frac{1-t^m}{1-t}dt \notag \\
&= \int_0^1 \frac{\ln(2)+ \ln(1- \frac{t}{2})}{1-t}dt \notag \\
&= \int_0^1 \frac{\ln(2 - t )}{1-t}dt = \int_0^1 \frac{\ln(1+x )}{x}dx 
= \frac{\pi^2}{12} \notag \\
\notag
\end{align}

We are now ready to compute the variance of $ \tau_{M^n}$ in an unified manner, note that in \cite{Mag-Cou} using different  stationary phase  method, we where only able to compute the equivalent up to some constant e.g. Proposition   22, 24, and 26 in \cite{Mag-Cou}.

\begin{pro}\label{var_M}
If $ M^n \neq  \mathbb{P}^n(\mathbb{R})$ then  
\begin{align*}
Var[\tau_{M^n}] &= \frac{1}{\delta^2} \Big(\frac{\pi^2}{3}      - \sum_{k=1}^{\delta} \frac{1}{k^2} - \frac{2}{\delta} \sum_{k=1}^{\delta} \frac{1}{k}   \Big) \notag \\
&=\frac{\pi^2}{6\delta^2} -\frac{2\ln(\delta)}{\delta^3} + \frac{1-2\gamma_e}{ \delta^3} + o(\frac{1}{\delta^3}) \notag \\
& \sim \frac{\pi^2}{6\delta^2} \\
\end{align*}
 where $ \gamma_e$ is the Euler constant and
$$ \delta =  \alpha + \beta +1 .$$ 
\end{pro}
\begin{proof}
Using Proposition \ref{mean_M} and \eqref{C11} we have 
\begin{align*}
Var[\tau_{M^n}] &= \EE[\tau_{M^n}^2]- \EE[\tau_{M^n}]^2 \\
&= \frac{1}{\delta^2} \Big(\frac{\pi^2}{3} + 2\sum_{k=1}^{\delta-1} \frac{1}{k} \sum_{m=1}^{k-1} \frac{1}{m}     - \frac{2}{\delta^2}  - (\sum_{k=1}^{\delta} \frac1k ) ^2 \Big)\\
&= \frac{1}{\delta^2} \Big(\frac{\pi^2}{3} + 2\sum_{k=1}^{\delta-1} \frac{1}{k} \sum_{m=1}^{k-1} \frac{1}{m}     - \frac{2}{\delta^2}  - (\sum_{k=1}^{\delta-1} \frac1k + \frac{1}{\delta}) ^2 \Big)\\
&= \frac{1}{\delta^2} \Big(\frac{\pi^2}{3}    - \frac{2}{\delta^2}  - \sum_{k=1}^{\delta} \frac{1}{k^2} - \frac{2H_{\delta-1}}{\delta}\Big)\\
&= \frac{1}{\delta^2} \Big(\frac{\pi^2}{3}      - \sum_{k=1}^{\delta} \frac{1}{k^2} - \frac{2H_{\delta}}{\delta}\Big)\\
\end{align*}
Since 
$$  \sum_{k=1}^{\delta} \frac{1}{k^2} = \frac{\pi^2}{6} - \frac{1}{\delta} + o(\frac{1}{\delta}) , $$
and 
$$ \frac{H_{\delta}}{\delta} = \frac{\ln(\delta)}{\delta} + \frac{\gamma_e}{\delta}+ o(\frac{1}{\delta}) ,$$
where $ \gamma_e$ is the Euler constant, it follows that :
$$ Var[\tau_{M^n}] = \frac{\pi^2}{6\delta^2} -\frac{2\ln(\delta)}{\delta^3} + \frac{1-2\gamma_e}{ \delta^3} + o(\frac{1}{\delta^3}).$$

\end{proof}
\begin{rem}\label{Remunif}
All the asymptotic development could be known using the existing results for $H_{\delta}$ and $\sum_{k=1}^{\delta} \frac{1}{k^2} $. 
Moreover Propositions \ref{mean_M} and \ref{var_M} show that $ Var(\tau_{M^n}) = o( (\EE[\tau_{M^n} ]^2)$, which is enough to shows the cutoff see \cite{Mag-Cou}, and justify that  at time $ t^{M^n}_n(c) = \EE[\tau_{M^n}] + \frac{c}{n}  $ something could happen. 
\end{rem}

\subsection{The profile function for spheres and projective spaces}

In the following section we will compute the limiting profile as well as optimal windows for cutoff in separation for family of Brownian motion in sphere, and projective spaces.

\begin{theo}  \label{profile-S} For $M^n =  \mathbb{S}^n$, the separation discrepancy to equilibrium for the Brownian motion has the following asymptotic profile: for all $c \in \mathbb{R} $  we have  
$$\lim_{n \to \infty}  \fs(\cL(X_n( \frac{\ln(n)}{n} + \frac{c}{n}  )),\cU_{\mathbb{S}^n})  =    1- e^{- e^{-c}},$$
and the windows of the cutoff sequence  $(\frac{\ln(n)}{n},\frac{1}{n})$ is strongly optimal in the sens of definition \ref{def2}. Moreover the above convergence is uniform on all compact.
\end{theo}
\begin{proof}

Let $M^n =  \mathbb{S}^n   $, it follows from asymptotic study of the two first moments of $ \tau_{M^n}$ in \cite{Mag-Cou} that their are a cutoff phenomenon with cutoff time $ \frac{\ln(n)}{n}$ with  windows $ \frac{1}{n}$. 
So let $c \in \mathbb{R} $, and let $ t^{\mathbb{S}}_n(c) = \frac{\ln(n)}{n} + \frac{c}{n} $.

We are interested in the profile of the cutoff, hence we are interested in the asymptotic behavior in $n$ of 
$$ \fs(\cL(X_n(  t^{\mathbb{S}}_n(c) )),\cU_{M^n}) =  \mathbb{P} [ \tau_{M^n}\ge t^{\mathbb{S}}_n(c)  ]  .$$
We then want to pass to a limit in $n$ in \eqref{queue} when $ x =t^{\mathbb{S}}_n(c) $.

From  \eqref{queue},  let 
\begin{align*} J^{\mathbb{S}}_k(n) &=   (-1)^{k+1}  e^{\lambda_k  t^{\mathbb{S}}_n(c) } \mathcal{C}_k(\mathbb{S}^n,\alpha,\beta) . \\
\end{align*} 
Recall that for the sphere case $\alpha = \frac{n-2}{2} = \beta $ and so  $ \lambda_k = -k(k + n -1 ) $,  from  the definition of $\mathcal{C}_k(M^n,\alpha,\beta)$ in Theorem \ref{moment}, and for $ n  $ large enough such that 
$ \frac{\ln(n)}{n} + \frac{c}{n} \le 1 $
\begin{align} 
\vert J^{\mathbb{S}}_k(n) \vert &=   e^{ -k(k + n -1 ) ( \frac{\ln(n)}{n} + \frac{c}{n} )   }  \frac{ (2k + n- 1  ) \Gamma (k+ n - 1 )}{ k ! \Gamma ( n )}  \notag \\
&\le e^k  e^{ -k(k + n) ( \frac{\ln(n)}{n} + \frac{c}{n} )   }  \frac{ (2k) ( n- 1  ) \Gamma (k+ n - 1 )}{ k ! \Gamma ( n )} \notag  \\
&= 2e^{k}  e^{ -k(k + n) ( \frac{\ln(n)}{n} + \frac{c}{n} )   }  \frac{  \Gamma (k+ n - 1 )}{ (k-1) ! \Gamma ( n -1 )} \notag \\
&= \frac{2e^{k -kc}}{(k-1) !}   e^{ -k^2 ( \frac{\ln(n)}{n} + \frac{c}{n} )   } \frac{1}{n^k}  \frac{  \Gamma (k+ n - 1 )}{  \Gamma ( n -1 )} \label{unifn} \\
&= \frac{2e^{k -kc}}{(k-1) !}   e^{ -k^2 ( \frac{\ln(n)}{n} + \frac{c}{n} )   } \frac{(k+n-2 ) ...(n-1 )}{n^k}  \notag \\
&= \frac{2e^{k -kc}}{(k-1) !}   e^{ -k^2 ( \frac{\ln(n)}{n} + \frac{c}{n} )   } 
(1+ \frac{k-2}{n} )...(1- \frac{1}{n} ) \notag  \\
&\le \frac{2e^{k -kc}}{(k-1) !}   e^{ -k^2 ( \frac{\ln(n)}{n} + \frac{c}{n} )   } 
(1+ \frac{k-2}{n} )...(1+ \frac{1}{n})  \notag  \\
&= \frac{2e^{k -kc}}{(k-1) !}   e^{ -k^2 ( \frac{\ln(n)}{n} + \frac{c}{n} )   } 
e^{\sum_{j=1}^{k-2} \ln( 1 + \frac{j}{n})}  \notag  \\
&\le \frac{2e^{k -kc}}{(k-1) !}   e^{ -k^2 ( \frac{\ln(n)}{n} + \frac{c}{n} -\frac{1}{2n})   }   \quad, \text{ since }  \forall x\ge 0 \quad \ln(1+x)\le x \notag  \\
& \le \frac{2e^{k(1-c)}}{(k-1) !} \quad, \text{ for n large enough},   \notag  \\
\notag 
\end{align}
hence $J^{\mathbb{S}}_k(n) $ is uniformly bounded by an integrable series. 
Also for fixed $k$,
\begin{align} 
\vert J^{\mathbb{S}}_k(n) \vert & \sim_{n \to \infty } \frac{1}{k!}  e^{ -k( \ln(n) +c  )   }  \frac{  \Gamma (k+ n - 1 )}{  \Gamma ( n-1 )}  \notag \\
&\sim_{n \to \infty } \frac{e^{ -kc }}{k!}    \frac{  \Gamma (k+ n - 1 )}{ n^k \Gamma ( n-1 )}  \notag \\
&\sim_{n \to \infty } \frac{e^{ -kc }}{k!}    \frac{ (k+ n - 1 )^{(k+ n - 1 - \frac12)} e^{-(k+ n - 1 )} }   { n^k ( n-1 )^{ n-1 -\frac12} e^{-(n - 1 )}}  \notag\\
&\sim_{n \to \infty } \frac{e^{ -kc }}{k!}    \label{limn}  .\\ \notag
\end{align}
It follows from \eqref{unifn}, \eqref{limn} and  Lebesgue's dominated convergence theorem that
\begin{align*}
\lim_{n \to \infty}  \fs(\cL(X_n(  t^{\mathbb{S}}_n(c) )),\cU_{M^n})  &= \lim_{n \to \infty}  \mathbb{P} [ \tau_{M^n} \ge t^{\mathbb{S}}_n(c)  ] \\
&=  \lim_{n \to \infty}    \sum_{k=1}^{\infty} (-1)^{k+1}  e^{\lambda_k  t^{\mathbb{S}}_n(c)} \mathcal{C}_k(M^n,\alpha,\beta)\\
&=  \sum_{k=1}^{\infty} (-1)^{k+1} \frac{e^{ -kc }}{k!}\\
&= 1- e^{- e^{-c}}.\\
\end{align*}
The above computation also shows that the windows $ \frac1n$ is in fact strongly optimal in the sens of definition \ref{def2}.
For the uniformity in all compact, let $c \in [-A,A]$,  
\begin{align}
 \vert J^{\mathbb{S}}_k(n) - (-1)^{k+1} \frac{e^{ -kc }}{k!}\vert &=  \vert e^{ -k(k + n -1 ) ( \frac{\ln(n)}{n} + \frac{c}{n} )   }  \frac{ (2k + n- 1  ) \Gamma (k+ n - 1 )}{ k ! \Gamma ( n )}   -  \frac{e^{ -kc }}{k!} \vert  \notag \\
 &=  \frac{e^{ -kc }}{k!} \vert  e^{ -k(k -1 ) ( \frac{\ln(n)}{n} + \frac{c}{n} )  } \frac{ (2k + n- 1  ) \Gamma (k+ n - 1 )}{ n^k \Gamma ( n )} -1  \vert. \label{UNI}\\
 \notag
\end{align}
Since  for all $x\ge 0$ 
$$ x- \frac{x^2}{2} \le \ln(1+x) \le x, $$
we have as in \eqref{unifn} that for $k\ge 2$:   $$ \frac{ (2k + n- 1  ) \Gamma (k+ n - 1 )}{ n^k \Gamma ( n )} =  \frac{(n+ 2k -1 ) (n+k-2) ...(n)}{n^k} \le e^{\frac{k(k-1)}{2n} +\frac{k}{n}} \le   e^{k(k-1)\frac{ 3}{2n}} ,$$
and
$$ \frac{ (2k + n- 1  ) \Gamma (k+ n - 1 )}{ n^k \Gamma ( n )} \ge  \frac{(n+ k -1 ) (n+k-2) ...(n)}{n^k} \ge e^{\frac{k(k-1)}{2n} -\frac{1}{2n^2}\sum_{j=1}^{k-1} j^2} \ge e^{\frac{k(k-1)}{2n} -\frac{1}{2n^2}(k-1)^3}.$$
It follows that for $ k\ge 2$  and for $ n$ large enough and  uniformly in $ c \in [-A,A],$ (such that $ \frac{\ln(n)}{n} + \frac{c}{n}  - \frac{3}{2n}  \ge  \frac{\ln(n)}{n} - \frac{A}{n}  - \frac{3}{2n} \ge 0 $)
$$ e^{ -k(k -1 ) ( \frac{\ln(n)}{n} + \frac{c}{n} )  } \frac{ (2k + n- 1  ) \Gamma (k+ n - 1 )}{ n^k \Gamma ( n )} \le e^{ -k(k -1 ) ( \frac{\ln(n)}{n} + \frac{c}{n}  - \frac{3}{2n})  }  \le 1 $$ 
and so since for $x\ge 0 $, $ \vert e^{-x} -1  \vert \le x $ 
\begin{align}
 \vert  e^{ -k(k -1 ) ( \frac{\ln(n)}{n} + \frac{c}{n} )  } \frac{ (2k + n- 1  ) \Gamma (k+ n - 1 )}{ n^k \Gamma ( n )} -1  \vert & \le  \vert e^{ -k(k -1 ) ( \frac{\ln(n)}{n} + \frac{c}{n} - \frac{1}{2n})-  \frac{1}{2n^2}(k-1)^3 } -1 \vert \notag \\
 & \le k(k -1 ) ( \frac{\ln(n)}{n} + \frac{c}{n} - \frac{1}{2n}) +  \frac{1}{2n^2}(k-1)^3 \notag \\
 & \le k(k -1 ) ( \frac{\ln(n)}{n} + \frac{A}{n} - \frac{1}{2n}) +  \frac{1}{2n^2}(k-1)^3 . \label{UNIF1}\\
 \notag
\end{align}
Since for  $k=1$, $ \vert J^{\mathbb{S}}_1(n) -  \frac{e^{ -c }}{1!}\vert \le \frac{e^{ -c }}{1!} \vert   \frac{n+1}{n} - 1\vert ,$ the uniform convergence on all compact follows from \eqref{UNI} and \eqref{UNIF1}.

\end{proof}   

\begin{rem}
The above computation shows that we have the following convergence in law :
\begin{align*}
  n {\tau_{\mathbb{S}^n} } -  \ln(n)\underset{n \to +\infty}{\overset{\mathcal{L}}{\longrightarrow}} Gumbel \big(0,1\big)
\end{align*}
where $Gumbel $ is the  Gumbel distribution.
\end{rem}

\begin{theo} \label{profile-P}
For $M^n =   \mathbb{P}^n(\mathbb{C}), \mathbb{P}^n(\mathbb{H}) \text{ or }  \mathbb{P}^n(\mathbb{R}) $, the separation discrepancy to equilibrium for the Brownian motion has the following asymptotic profile: for all $c \in \mathbb{R} $  we have  
$$\lim_{n \to \infty}  \fs(\cL(X_n( \frac{2 \ln(n)}{n} + \frac{c}{n} ) ,\cU_{M^n})  =    1- e^{- \frac{e^{-\frac{c}{2}}}{2}},$$
and the windows of the cutoff sequence  $(\frac{2\ln(n)}{n},\frac{1}{n})$ is strongly optimal in the sens of definition \ref{def2}. Moreover the above convergence is uniform on all compact.
\end{theo}
\begin{proof}
The proof is very similar to the above proof.
Let $M^n = \mathbb{P}^n(\mathbb{C})$, it follows from asymptotic study of the two first moments of $ \tau_{M^n}$ as in \cite{Mag-Cou} that their are a cutoff phenomenon with cutoff time $ \frac{2 \ln(n)}{n}$ with  windows $ \frac{1}{n}$, for $c \in \mathbb{R} $, let $ t^{\mathbb{P}(\mathbb{C})}_n(c) = \frac{2\ln(n)}{n} + \frac{c}{n} $, let
\begin{align*} J^{ \mathbb{P}(\mathbb{C}) }_k(n) &=   (-1)^{k+1}  e^{\lambda_k  t^{\mathbb{P}(\mathbb{C})}_n(c) } \mathcal{C}_k( \mathbb{P}^n(\mathbb{C}),\alpha,\beta) . \\
\end{align*} 
Recall that for the complex projective  case $\alpha = \frac{n-2}{2} $ and $  \beta = 0 $ and so  $ \lambda_k = -k(k + \frac{n}{2}  ) $,  using  Theorem \ref{moment} and the same computation as \eqref{unifn} we get that  $J^{\mathbb{P}(\mathbb{C}) }_k(n) $ is uniformly bounded by an integrable series, and for
 fixed $k$,
\begin{align} 
\vert J^{ \mathbb{P}(\mathbb{C})}_k(n) \vert &=    e^{ -k(k + \frac{n}{2})( \frac{2\ln(n)}{n} + \frac{c}{n}   ) } \frac{ (2k + \frac{n}{2} ) \Gamma (k+ \frac{n}{2})}{ k ! \Gamma ( \frac{n}{2} +1)}  \notag   \\
&\sim_{n \to \infty } \frac{1}{k!}  e^{ -k( \ln(n) +\frac{c}{2}  )   }  \frac{  \Gamma (k+ \frac{n}{2} )}{  \Gamma ( \frac{n}{2} )}  \notag \\
&\sim_{n \to \infty } \frac{e^{ -\frac{kc}{2} }}{2^k k!}    \frac{ \Gamma (k+ \frac{n}{2} )}{ (\frac{n}{2})^k \Gamma ( \frac{n}{2} )}  \notag \\
&\sim_{n \to \infty } \frac{e^{ -\frac{kc}{2} }}{2^k k!}    \frac{(k+ \frac{n}{2}  )^{(k+ \frac{n}{2}  - \frac12)} e^{-(k+ \frac{n}{2} )} }{ (\frac{n}{2})^k  ( \frac{n}{2} )^{ \frac{n}{2} -\frac12} e^{-(\frac{n}{2} )}       }  \notag \\
&\sim_{n \to \infty } \frac{e^{ -\frac{kc}{2} }}{2^k k!}    \notag  .\\ \notag
\end{align}
It follows that 
$$\lim_{n \to \infty}  \fs(\cL(X_n(t^{\mathbb{P}(\mathbb{C})}_n(c)  )),\cU_{\mathbb{P}^n(\mathbb{C})})  =    1- e^{- \frac{e^{-\frac{c}{2}}}{2}}.$$

Recall that for the quaternionic projective  case $\alpha = \frac{n-2}{2} $ and $  \beta = 1$ and so  $ \lambda_k = -k(k + \frac{n}{2} +1 ) $ and the cutoff time is the same as the complex projective case \cite{Mag-Cou}, let $ t^{\mathbb{P}(\mathbb{H})}_n(c) = \frac{2\ln(n)}{n} + \frac{c}{n} $ and we have in a similar way
$$\lim_{n \to \infty}  \fs(\cL(X_n(t^{\mathbb{P}(\mathbb{H})}_n(c)  )),\cU_{\mathbb{P}^n(\mathbb{H})})  =    1- e^{- \frac{e^{-\frac{c}{2}}}{2}}.$$

Let $M^n = \mathbb{P}^n(\mathbb{R})$, from  \cite{Mag-Cou} their are a cutoff phenomenon with cutoff time  $ \frac{2 \ln(n)}{n}$ with optimal windows $ \frac{1}{n}$, for $c \in \mathbb{R} $, let $ t^{\mathbb{P}(\mathbb{R})}_n(c) = \frac{2\ln(n)}{n} + \frac{c}{n} $, and let
\begin{align*} J^{ \mathbb{P}(\mathbb{R}) }_k(n) &=   (-1)^{k+1}  e^{\lambda_k  t^{\mathbb{P}(\mathbb{R})}_n(c) } \mathcal{C}_k( \mathbb{P}^n(\mathbb{R}),\alpha,\beta) . \\
\end{align*} 
Recall that for the complex projective  case $\alpha = \frac{n-2}{2} = \beta  $ and so  $ \lambda_k = -k(k + \frac{n-1}{2}  ) $ (recall that the diameter is $\pi$ ),  using  Theorem \ref{moment} and the same computation as \eqref{unifn} we get that  $J^{\mathbb{P}(\mathbb{R}) }_k(n) $ is uniformly bounded by an integrable series, and for  fixed $k$,
\begin{align} 
\vert J^{ \mathbb{P}(\mathbb{R})}_k(n) \vert &=    e^{ -k(k + \frac{n-1}{2})( \frac{2\ln(n)}{n} + \frac{c}{n}   ) }    \frac{(4k + n-1  ) \Gamma (2k+  n-1  )\Gamma(\frac{n}{2})}{ 2^{2k} k! \Gamma (k+  \frac{n}{2}) \Gamma (n)}    \notag   \\
&\sim_{n \to \infty } \frac{ e^{ -\frac{kc}{2} }}{2^{2k} k!}    
  \frac{ \Gamma (2k+  n-1  )\Gamma(\frac{n}{2})}{ n^k  \Gamma (n -1 ) \Gamma (k+  \frac{n}{2})}   \notag \\
&\sim_{n \to \infty } \frac{ e^{ -\frac{kc}{2} }}{2^{2k} k!}    
  \frac{ (2k+  n-1  )^{2k+  n-1 - \frac12 } e^{-(2k+  n-1  )} (\frac{n}{2})^{\frac{n}{2} - \frac12 } e^{- \frac{n}{2}}  }{ n^k   (n -1 )^{n -1 - \frac12 } e^{-(n -1 )}   (k+  \frac{n}{2})^{ k+  \frac{n}{2} - \frac12 }  e^{-(k+  \frac{n}{2}) }}   \notag \\
  &\sim_{n \to \infty } \frac{ e^{ -\frac{kc}{2} }}{2^{2k} k!}    
  \frac{ (2k+  n-1  )^{2k+  n-1 - \frac12 } e^{-(k  )} (\frac{n}{2})^{\frac{n}{2} - \frac12 } }{ n^k   (n -1 )^{n -1 - \frac12 }  (k+  \frac{n}{2})^{ k+  \frac{n}{2} - \frac12 }  }   \notag \\
&\sim_{n \to \infty } \frac{e^{ -\frac{kc}{2} }}{2^k k!}    \notag  .\\ \notag
\end{align}
It follows that 
$$\lim_{n \to \infty}  \fs(\cL(X_n(t^{\mathbb{P}(\mathbb{R})}_n(c)  )),\cU_{\mathbb{P}^n(\mathbb{R})})  =    1- e^{- \frac{e^{-\frac{c}{2}}}{2} },$$
also the windows $ \frac1n$ is in fact strongly optimal. The uniform convergence in all compact follows with the same proof as in Theorem \ref{profile-S}.
\end{proof}   

\begin{rem}
If $M^n =   \mathbb{P}^n(\mathbb{C}), \mathbb{P}^n(\mathbb{H}) \text{ or }  \mathbb{P}^n(\mathbb{R}) $ we have the following convergence in law :
\begin{align*}
  n {\tau_{M^n} } -  2\ln(n)\underset{n \to +\infty}{\overset{\mathcal{L}}{\longrightarrow}} Gumbel \big(-2 \ln(2), 2 \big).
\end{align*}

\end{rem}


 \bibliographystyle{plain}

\vskip2cm
\hskip70mm
\vbox{
\copy4
 \vskip5mm
 \copy5
  \vskip5mm
 \copy6
}

\end{document}